\documentclass[12pt]{iopart}

\usepackage{setstack}
\usepackage{amsthm}
\usepackage{graphicx}
\usepackage{amsfonts}
\usepackage{amssymb}
\usepackage{subfigure}
\usepackage{url}
\usepackage{enumerate}

\newtheorem{theorem}{Theorem}[section]

\newtheorem{corollary}[theorem]{Corollary}

\newtheorem{lemma}[theorem]{Lemma}

\newtheorem{proposition}[theorem]{Proposition}
\newtheorem{remark}[theorem]{Remark}

\begin{document}

\renewcommand{\underset}[2]{\ensuremath{\mathop{\kern\z@\mbox{#2}}\limits_{\mbox{\scriptsize #1}}}}

\maketitle

\title{Defects and boundary layers in non-Euclidean plates}
\author{J. A. Gemmer and S. C. Venkataramani}
\address{University of Arizona, Program in Applied Mathematics, 617 N. Santa Rita Ave. Tucson, AZ 85721}
\ead{jgemmer@math.arizona.edu}

\begin{abstract}
We investigate the behavior of non-Euclidean plates with constant negative Gaussian curvature using the F\"oppl-von K\'arm\'an reduced theory of elasticity. Motivated by recent experimental results, we focus on annuli with a periodic profile. We prove rigorous upper and lower bounds for the elastic energy that scales like the thickness squared. In particular we show that are only two types of global minimizers -- deformations that remain flat and saddle shaped deformations with isolated regions of stretching near the edge of the annulus. We also show that there exist local minimizers with a periodic profile that have additional boundary layers near their lines of inflection. These additional boundary layers are a new phenomenon in thin elastic sheets and are necessary to regularize jump discontinuities in the azimuthal curvature across lines of inflection. We rigorously derive scaling laws for the width of these boundary layers as a function of the thickness of the sheet.
 
\end{abstract}

\submitto{Nonlinearity}

%
%
\section{Introduction}
\subsection{Non-Euclidean model of plates}
Laterally swelling and shrinking thin elastic sheets are ubiquitous in nature and industry and are capable of forming complex surfaces of various geometries. The examples shown in figure \ref{shapes} include shapes formed by differential growth, thermal expansion, and the inhomogeneous swelling of hydrogels.  Hydrogels in particular have received a lot of attention because of their ability to simulate biological growth by differentially swelling when activated by external stimuli such as light \cite{light}, solvents \cite{bending&twisting}, warm water \cite{Klein2007Shape}, and pH \cite{pH}. The morphology of these structures is a result of the sheet buckling to relieve the residual stress caused by the material's resistance to cavitation and interpenetration of the material \cite{goriely2005growth}. 
\begin{figure}[htp] 
\begin{center}
\subfigure[]{
\includegraphics[width=2in,height=1.8in]{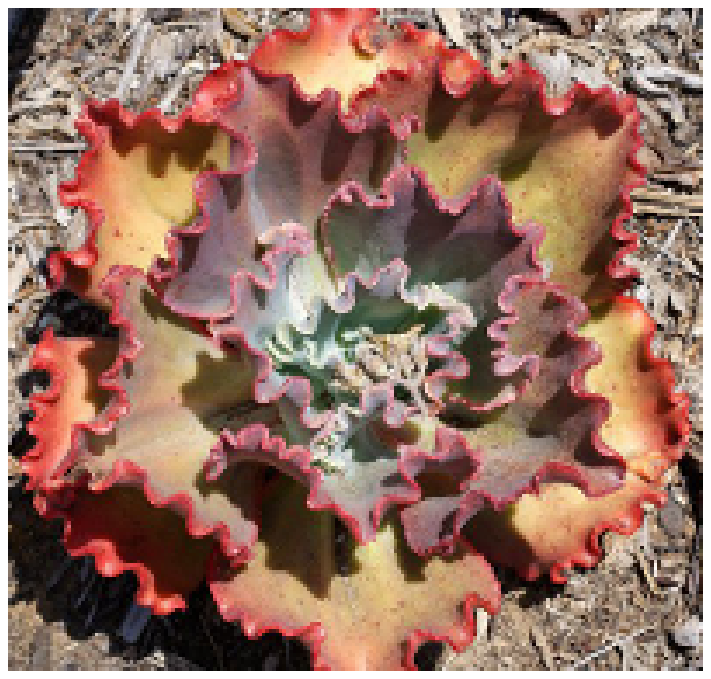} 
\label{shapes:lichen}
}
\subfigure[]{
\includegraphics[width=2in,height=1.8in]{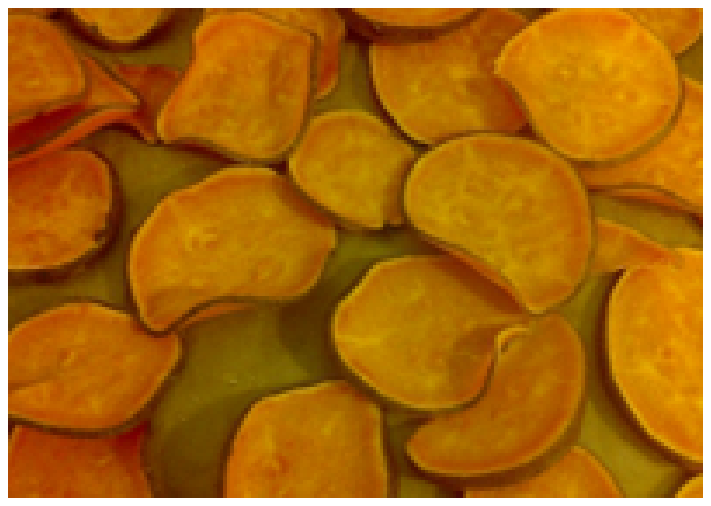}
\label{shapes:chips}
}
\subfigure[]{
\includegraphics[width=2in,height=1.8in]{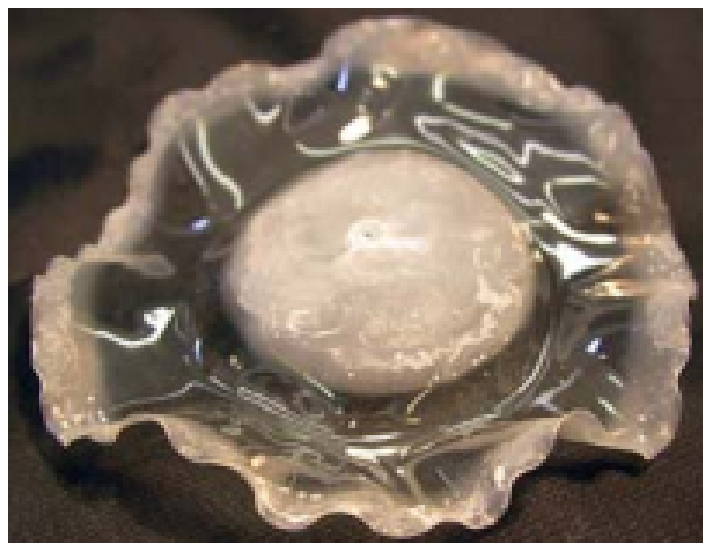}
\label{shapes:hydrogel}
}
\end{center}
\caption{\subref{shapes:lichen} A hybrid species of the Echeveria plant. The rippling near the boundary is caused by local differential growth and has been the focus of several works \cite{Marder2003Leaves, Audoly2002Ruban, Audoly2003SelfSim, HaiyiLiang12292009}. \subref{shapes:chips}  The heating and subsequent drying of potato chips generates shapes with a hyperbolic geometry. \subref{shapes:hydrogel} N-isopropylacrylamide (NIPA) hydrogel disk that has undergone controlled shrinking \cite{Klein2007Shape}. Hydrogels replicate many of the characteristic features present in differential growth and are excellent tools for studying such processes quantitatively.}
\label{shapes}
\end{figure}

One model of such swelling bodies hypothesizes that the equilibrium configuration is the minimum of a ``non-Euclidean'' free energy functional $E_{3\rmd}:W^{1,2}(\mathcal{D}_{3\rmd},\mathbb{R}^3)\rightarrow \mathbb{R}$ that measures strains from a fixed three-dimensional Riemannian metric $ \mathbf{g}_{3\rmd}$ defined on a simply connected domain $ \mathcal{D}_{3\rmd} \subset \mathbb{R}^3$ \cite{Audoly2003SelfSim, Marder2006Leaves, Efrati2007Buckling, marta2011gammalim}. Laterally swelling thin elastic sheets of thickness $t$ can be modelled in this framework by a two dimensional Riemannian metric $\mathbf{g}$ defined on the mid-surface of the sheet $ \mathcal{D}\subset \mathbb{R}^2$. That is, $\mathcal{D}_{3\rmd}$ can be decomposed as $\mathcal{D}_{3\rmd}=\mathcal{D}\times (-t/2,t/2)$ and in an appropriate coordinate system $\mathbf{g}_{3\rmd}$ is given by
\begin{equation*}
\mathbf{g}_{3\rmd}=\left(\begin{array}{cc}
\mathbf{g} & 0\\
0 & 1
\end{array}\right).
\end{equation*}

The non-Euclidean model of elasticity has grown out of an experiment performed by Sharon et. al. that studied the self-similar rippling patterns observed along the edges of torn elastic sheets and leaves \cite{Sharon2002BuckleCascades}. This rippling pattern was studied in the strip geometry by Marder et. al. \cite{Marder2003Leaves} and Audololy and Boudaoud \cite{Audoly2002Ruban, Audoly2003SelfSim} by using a metric that is localized to one edge of the strip. In particular, in \cite{Audoly2003SelfSim} it was shown numerically that the self similar patterns could be explained as approximate isometric immersions of $\mathbf{g}$ with the particular isometric immersion selected as the minimum of a bending energy functional.

\subsection{Reduced theories}
In many applications the thickness of such sheets is much smaller then the diameter of $ \mathcal{D}$. Consequently, there is considerable interest in obtaining reduced energy functionals $ E_t$ defined on a suitable space of mappings of the mid-surface $\mathcal{D}$ into $ \mathbb{R}^3$ such that minimizers of $ E_t$ approximate minimizers $E_{3\rmd}$ in an appropriate asymptotic limit of vanishing thickness. For the classical theory of plates, i.e. when $ \mathbf{g}$ is the Euclidean metric, a family of such reduced theories have been derived through the technique of $\Gamma$-convergence \cite{MuellerPlates}.  For non-Euclidean plates, there are two reduced theories that have recently been rigorously obtained as $ \Gamma$-limits by Lewicka and Pakzad \cite{marta2011gammalim} and Lewicka, Pakzad, and Mahadevan \cite{linearizedGeometry}.

\subsubsection{Kirchhoff model:} The first of these theories, which is sometimes called the \emph{Kirchhoff model}, states that if the set $\mathcal{A}_{Ki}$  of finite bending energy isometric immersions is not empty, that is $ \mathcal{A}_{\rm{Ki}}=\{\mathbf{x}\in   W^{2,2}(\mathcal{D},\mathbb{R}^3):  (D\mathbf{x})^T\cdot D\mathbf{x}=\mathbf{g}\} \neq \emptyset$, then
\begin{equation*} \Gamma-\lim_{t\rightarrow 0}
\frac{1}{t^2}E_{3\rmd}=E_{\rm{Ki}},
\end{equation*}
with the curvature functional $ E_{\rm{Ki}}:W^{2,2}(\mathcal{D},\mathbb{R}^3)\rightarrow \mathbb{R}$ defined by
\begin{equation*}
E_{\rm{Ki}}[\mathbf{x}]=\left\{\begin{array}{cc}
\displaystyle{\frac{Y}{24(1+\nu)}\int_{\mathcal{D}}\left[\frac{4H^2}{1-\nu}-2K\right]\rmd A_{\mathbf{g}}} & \text{ if }  \mathbf{x}\in \mathcal{A}_{\rm{Ki}}\\
\infty & \text{ if } \mathbf{x} \notin \mathcal{A}_{\rm{Ki}}
\end{array}\right.,
\end{equation*}
where $ Y $ and $ \nu $ are the Young's modulus and Poisson ratio of the material respectively, $ H $ and $ K $ mean and Gaussian curvatures of the surface $ \mathbf{x}(\mathcal{D}) $ respectively, and $ \rmd A_{\mathbf{g}}$ the area form induced by $ \mathbf{g}$ \cite{marta2011gammalim}. This reduced theory captures the intuition that in the vanishing thickness limit the mid-surface should deform into an isometric immersion with a low amount of bending energy.

\subsubsection{F\"oppl-von K\'arm\'an Model:} The second reduced theory, which has been called the \emph{small slope approximation} or \emph{F\"oppl-von K\'arm\'an (FvK) ansatz}, has also been rigorously derived as a $\Gamma$-limit when the following assumptions are met:
\begin{enumerate}
\item The metric $ \mathbf{g} $ satisfies the following scaling $ \mathbf{g}=\mathbf{g}_0+t^2\mathbf{g}_1$, where $ \mathbf{g}_0$ is the Euclidean metric.
\item The deformation of the mid-surface $ \mathbf{x}:\mathcal{D}\rightarrow \mathbb{R} ^3$ satisfies an ``in-plane'' and ``out-of-plane'' decomposition of the form
\begin{equation*}
\mathbf{x}=i+t i_{\perp}\circ \eta +t^2 i\circ\chi,
\end{equation*}
where $\chi\in W^{1,2}(\mathcal{D},\mathbb{R}^2)$, $\eta\in W^{2,2}(\mathcal{D},\mathbb{R})$, $i:\mathbb{R}^2\rightarrow \mathbb{R}^3$ is the standard immersion and $i_{\perp}$ maps into the orthogonal compliment of $i(\mathbb{R}^2)$.
\end{enumerate}
 In this context it is natural to define the FvK admissible set by $ \mathcal{A}=W^{1,2}(\mathcal{D},\mathbb{R}^2)\times W^{2,2}(\mathcal{D},\mathbb{R})$. When we write that a deformation $ \mathbf{x}:\mathcal{D}\rightarrow \mathbb{R}^3 $ satisfies $ \mathbf{x}\in \mathcal{A} $ we in actuality mean that there exists $ (\chi,\eta)\in \mathcal{A} $ such that $ \mathbf{x}=i+t i_{\perp}\circ \eta +t^2 i\circ\chi$. The \textbf{in-plane strain tensor} $ \gamma:\mathcal{A}\rightarrow \mathbb{R}^{2\times2} $ is defined in this approximation by 
\begin{equation} \label{FvK:strain}
\gamma(\chi,\eta)=\left(D\chi\right)^T+D\chi+\left(D\eta\right)^T\cdot D\eta-\mathbf{g}_1,
\end{equation}
which measures to $ \Or(t^2) $ the deviation of $ \mathbf{x} $ from being an isometric immersion.
 
The reduced energy under the above assumptions is given by the following  $\Gamma$-limit:
\begin{equation*}
\Gamma-\lim_{t\rightarrow 0} \frac{1}{t^4}E_{3\rmd}=E_{\rm{FvK}},
\end{equation*}
with $E_{\rm{FvK}}:\mathcal{A}\rightarrow \mathbb{R}$ defined by
\begin{equation}\label{intro:FvKenergy}
E_{\rm{FvK}}[\mathbf{x}]=\frac{Y}{8(1+\nu)}\int_{\mathcal{D}}Q(\gamma)\,\rmd A+\frac{Y}{24(1+\nu)}\int_{\mathcal{D}}Q(D^2\eta)\,\rmd A,
\end{equation}
where $Q:\mathbb{R}^{2\times 2}\rightarrow \mathbb{R}$ is the quadratic form 
\begin{equation*}
Q(A)=(1-\nu)^{-1}\tr(A)^2-2\det(A),
\end{equation*}
$D^2\eta$ the Hessian matrix of second derivatives of $\eta$ and $ \rmd A $ the Euclidean area form \cite{linearizedGeometry}. This energy can be interpreted as approximating $E_{3\rmd}$ by an energy that is the additive decomposition of a \textbf{stretching energy functional} $S_{\rm{FvK}}:\mathcal{A}\rightarrow \mathbb{R}$ measuring the lowest order elastic energy of the in-plane components of the in-plane strain $ (D\mathbf{x})^T\cdot D\mathbf{x}-\mathbf{g} $ and a \textbf{bending energy functional} $ B_{\rm{FvK}}:\mathcal{A}\rightarrow \mathbb{R}$ measuring the lowest order approximation of $E_{\rm{Ki}}$. Additionally, in this approximation, the terms $ \Delta \eta $ and $ \det(D^2\eta) $ correspond to the mean and Gaussian curvatures respectfully. The Euler-Lagrange equations generated by the variation of $E_{\rm{FvK}}$ are a modified version of the classical FvK equations and have been used to model morphogenesis in soft tissue \cite{HaiyiLiang12292009, Dervaux(2008)}.

\subsection{Motivation for and physical implications of this study}
There have been several theoretical \cite{Santangelo-hyperbolicplane, Santangelo-MinimalRes, Gemmer2011} and experimental \cite{Shankar2011Gels} investigations that have studied non-Euclidean plates when $ \mathcal{D}$  is a disk or an annulus with a two dimensional metric $ \mathbf{g}_{K_0} $ that has corresponding constant negative Gaussian curvature $K_0$. Since $ \mathcal{A}_{\rm{Ki}}\neq \emptyset $ for these metrics and domains \cite{Poznyak1973, han&hong}, that is there exist exact isometric immersions with finite bending energy, it follows from the Kirchhoff model that in the vanishing thickness limit the mid-surface of the minimizers of $ E_{3\rmd} $ should converge to a fixed shape. Furthermore, in a related work Trejo, Ben Amar, and M\"uller \cite{Trejo-2009} have shown that surfaces with constant negative Gaussian curvature can be constructed from closed curves generating pseudospherical surfaces which has implications to the modeling of growing bodies with a tubular topology.

Experimentally, upon swelling, annular hydrogels with a fixed inner and outer radius and a prescribed two dimensional metric $ \mathbf{g}_{K_0} $ obtain a periodic profile with $n$ waves.  The number of waves refines with decreasing thickness according to the scaling $ n\sim t^{-1/2} $ with the bending energy $ E_{\rm{Ki}} $ diverging according to the scaling $ E_{\rm{Ki}}\sim t^{-1} $ \cite{Shankar2011Gels}. Therefore, at least for the range of thicknesses considered in the  experiment it appears that the annuli are not converging  to a fixed shape. If we believe that the experimental configurations are global minimizers of the elastic energy, then they should converge weakly to an isometric immersion with finite bending energy. It is this apparent contradiction between the experimental observations and the theoretical prediction that motivates the study in this paper.

A possible explanation for this apparent disagreement between theory and experiment is that the range of thicknesses which are accessible to the experiments corresponds to an intermediate asymptotic regime, {\em different from the regime of vanishing thickness}. That is, although the number of waves increases in the intermediate regime, for smaller thicknesses the system cross over into a different regime at which point the shape (and hence also the number of waves) will stabilize and converge to an exact isometric immersion.

To study this question, we need to determine the lowest energy configurations of the sheet not only in the limit $t \to 0$, but also for all $t > 0$. In addition to the thickness, a second relevant parameter in this system is the dimensionless curvature of the sheet $\epsilon$, namely the curvature of the sheet normalized by its radius. For the experiments, this parameter is smaller than 1. Thus it is appropriate to study these sheets in the F\"oppl-von Karman (geometrically linear) regime \cite{linearizedGeometry}.  

We obtain rigorous scaling laws for the F\"{o}ppl-von Karman elastic energy for all $t > 0$ and all configurations with $n$-waves. In particular, since we are interested in all $t > 0$, we do not  assume that the configurations are close to isometric immersions. Our scaling laws for the energy are {\em ansatz-free}, and applicable to all $n$-wave configurations. The rigorous bounds show that the $n = 2$ configurations have lower energy than all the $n > 2$ configurations {\em for all thicknesses $t > 0$}. This falsifies the idea that the experimentally observed $ n > 2$ configurations are global minimizers for the elastic energy.

In addition to the global minimizers ($n = 2$) we also determine the local minimizers ($n > 2$) of the elastic energy for $t > 0$. As $t  \rightarrow 0$, these local minimizers also converge to isometric immersions {\em which however are not smooth}. These isometric immersions do not self-intersect and for this metric there is no analogue of the skewed e-cones introduced and studied in \cite{econes}. But, all the local minimizers with $n > 2$ waves converge to configurations with singularities (``lines of inflection'') corresponding to a continuous tangent plane, but a  jump in the curvature \cite{Gemmer2011}. This should be contrasted with the elastic ridge singularity in crumpled sheets which mediates a jump in the tangent plane \cite{Lobkovsky1996}.

 We have determined the shape of the local minimizers for $t > 0$ by solving for the boundary layers which resolve these jumps in curvature. For $n > 2$, we identify two types of overlapping boundary layers  corresponding to reductions in the Gauss curvature and the mean curvature respectively. This overlapping is, to our knowledge a unique phenomenon and one that is potentially observable in experiments.

Finally, we observe that although the $n > 2$ local minimizers are periodic, the fact that they have singularities (or narrow boundary layers) implies that they do not have a sparse representation in terms of a Fourier basis. Thus, these structures will be difficult to analyze numerically using a Fourier basis expansion as was done for the strip geometry by Audoly and Boudaoud \cite{Audoly2003SelfSim}.

\subsection{Structure of paper and main results}

In this section we describe our results in a mathematically precise manner and also outline the structure of the paper. As discussed above, we study the convergence and scaling of local and global minimizers of the F\"oppl-von K\'arm\'an energy with the metric $ \mathbf{g}_{K_0} $ for decreasing values of $ t  > 0$. A major point of this paper is that we prove that in the FvK approximation that for all thicknesses saddle shapes ($n = 2$) are energetically preferred over deformations with more waves ($ n > 2$). 

Motivated by the experimental results in \cite{Shankar2011Gels} we focus on studying minimizers with a periodic profile when $ \mathcal{D}$ is an annulus with inner and outer radius $ \rho_0 $ and $ R $ respectively. To be precise, for $ n\in \{2,3\ldots\} $, we study minimizers over the admissible set of $\mathbf{n}$\textbf{-periodic deformations} $\mathcal{A}_n\subset \mathcal{A}$ defined by $(\chi,\eta)\in \mathcal{A}_n$ if and only if in polar coordinates $ (\rho,\theta) $ the out-of-plane displacement $\eta$ satisfies
\begin{enumerate}
\item $\eta$ is periodic in $\theta$ with period $2\pi/n$,
\item $\eta$ vanishes along the lines $\theta=0$ and $\theta=\pi/n$,
\item $\eta\left(\theta-\pi/n\right)=-\eta(\theta)$.
\end{enumerate}
We call the lines $ \theta=m\pi/n $, $ m\in \{0,\ldots,2n-1\} $, \textbf{lines of inflection}. 

In section 2 we construct the FvK elastic energy by expanding $\mathbf{g}_{K_0} $ in the \textbf{dimensionless curvature} $\epsilon$ defined by
\begin{equation}\label{intro:epsilon}
\epsilon=\sqrt{-K_0}R.
\end{equation}
We repose the variational as the minimum of a dimensionless functional  $\mathcal{E}_{\tau}=\mathcal{S}+\tau^2\mathcal{B}$, with the \textbf{dimensionless thickness} $\tau$ defined by
\begin{equation}\label{intro:tau}
\tau=t/(\sqrt{3}R\epsilon).
\end{equation}
We further derive the Euler-Lagrange equations and natural boundary conditions corresponding to the variation of $\mathcal{E}_{\tau}$.

In sections 3 and 4 we construct minimizers over the admissible sets $\mathcal{A}_{f}=\{\mathbf{x}\in \mathcal{A}: \mathcal{B}[\mathbf{x}]=0\}$ and $ \mathcal{A}_{0}=\{\mathbf{x}\in \mathcal{A}: \mathcal{S}[\mathbf{x}]=0\}$. The set $ \mathcal{A}_f $ consists of flat deformations and the minimum over $\mathcal{A}_f$ is obtained by solving the Euler-Lagrange equations and its natural boundary conditions with $\eta=0 $. The set $ \mathcal{A}_0 $ consists of isometric immersions in the F\"oppl-von K\'arm\'an approximation and can be obtained by solving the Monge-Ampere equation 
\begin{equation} \label{intro:Monge-Ampere}
\det(D^2\eta)=-1.
\end{equation}
The global minimum of $ \mathcal{B} $ over solutions to (\ref{intro:Monge-Ampere}) are obtained by the approximate minimal surfaces $ \eta=xy $ and thus following the results in \cite{Gemmer2011} we have the following theorem:
\begin{theorem}
Let $\mathbf{x}^*_{\tau}\in \mathcal{A}$ be a sequence with corresponding out-of-plane displacement $\eta^*_{\tau}$ such that $\inf_{\mathbf{x}\in \mathcal{A}}\mathcal{E}_{\tau}[\mathbf{x}]=\mathcal{E}_{\tau}[\mathbf{x}^*_{\tau}]$. Then,
\begin{enumerate}
\item $\displaystyle{
\lim_{\tau\rightarrow 0}\inf_{\mathbf{x}\in \mathcal{A}}\frac{1}{\tau^2}\mathcal{E}_{\tau}[\mathbf{x}]=\lim_{\tau\rightarrow 0}\frac{1}{\tau^2}\mathcal{E}_{\tau}[\mathbf{x^*_{\tau}}]=\inf_{\mathbf{x}\in \mathcal{A}_0}\mathcal{B}[\mathbf{x}]=2\pi(1-\rho_0^2R^{-2}).} $
\item There exists a subsequence $ \eta^*_{\tau_k}$ and $ \mathbf{x}^*\in \mathcal{A}_0 $ with corresponding out-of-plane displacement $ \eta^* $ such that $\eta^*_{\tau_k}\rightharpoonup \eta^*$. Moreover, there exists $ A\in SO(2) $ and $ b\in \mathbb{R} $ such that $ \eta^*(A(x,y))+b=xy$.
\end{enumerate} 
\end{theorem}
\noindent This theorem proves that in the vanishing thickness limit saddle shaped deformations are energetically preferred.

Following the construction in \cite{Gemmer2011}, we show that $ \mathcal{A}_n\cap \mathcal{A}_0\neq \emptyset $, that is there exists $n$-periodic isometric immersions,  and using these constructions as test functions we show the following result:
\begin{lemma} \label{intro:upperbound} Let $\mathbf{x}^*\in \mathcal{A}_n$ such that $\mathcal{E}_{\tau}[\mathbf{x}^*]=\inf_{\mathbf{x}\in \mathcal{A}_n}\mathcal{E}_{\tau}[\mathbf{x}]$, then 
\begin{equation*}
\mathcal{E}_{\tau}[\mathbf{x}^*]\leq \min\left \{ Cn^2\tau^2,\mathcal{F}\right\},
\end{equation*}
where $\mathcal{F}=\inf_{\mathbf{x}\in \mathcal{A}_f}\mathcal{E}_{\tau}[\mathbf{x}] $ and $ C $ is a constant independent of $ n $ and $ \tau $.
\end{lemma}
\noindent This lemma captures the difference in scaling of $ \mathcal{E}_{\tau} $ with $ \tau $ for flat deformations and isometric immersions by using specific test functions on $ \mathcal{E}_{\tau}$. Furthermore, this lemma illustrates the growth in $n$ of the total energy for $n$-periodic isometric immersions. 

In section 5 we numerically minimize $ \mathcal{E}_{\tau} $ over $ \mathcal{A}_n $ using a Rayleigh-Ritz type method. The results of the numerics indicate that minimizers transition from flat deformations to buckled shapes such that with decreasing thickness the deformations converge to $ \mathbf{x} \in \mathcal{A}_n\cap \mathcal{A}_0$. Moreover, the results of the numerics indicate that for all thickness values the $2$-wave saddle shape is energetically preferred over profiles with a higher number of waves. In the bulk of the domain the stretching energy of these buckled shapes is approximately zero with regions near the edges and, for $ n\geq 3 $, along the lines of inflection in which stretching energy is concentrated. These results indicate that with decreasing thickness the stretching energy is concentrated in shrinking boundary layers in which the bending energy of the isometric immersion is reduced by adding a small localized amount of stretching.

In section 6 we improve on lemma \ref{intro:upperbound} and rigorously prove ansatz free lower bounds on $ \mathcal{E}_{\tau} $. Specifically, we prove the following theorem:
\begin{theorem}
Suppose $n\in \{2,3,\ldots\}$ and $\tau>0$. There exists constants $c$ and $C$ independent of $n$ such that
\begin{equation*}
 \min\left\{cn\tau^2,\frac{\mathcal{F}}{2}\right\}\leq \inf_{\mathbf{x}\in \mathcal{A}_n}\mathcal{E}[\mathbf{x}]\leq \min\{ Cn^2\tau^2,\mathcal{F}\}.
\end{equation*}
\end{theorem}
\noindent This theorem confirms the numerical results that with decreasing thickness the elastic energy of the minimizers scales like that of an isometric immersion. Furthermore, since both the upper and lower bounds in the theorem grow with $ n $ it quantifies how the elastic energy is penalized by adding more waves to a deformation. Additionally, this theorem proves that for all thicknesses the $2$-wave saddle shape is energetically preferred.

In section 7 we determine the scaling with $ \tau $ of the width of the boundary layers in which the stretching energy is significantly non-zero. Near the edges of the annulus the out of plane displacement can be additively decomposed into two terms that separately lower contributions to the bending energy from the mean and Gaussian curvatures. The width of the regions in which these terms are relevant satisfies the following scaling:
\begin{enumerate}
\item The boundary layer in which one term lowers the magnitude of the \emph{Gaussian curvature} has the following scaling
\begin{equation*}
\text{width}(\theta)_{\rho=\rho_0,R}\sim t^{\frac{1}{2}}|K_0|^{-1/4}\csc\left(\frac{\pi}{n}\right)\cos(\theta)\cos(\pi/n-\theta).
\end{equation*}
\item For $ n \geq 3 $, the boundary layer in which one term lowers the magnitude of the \emph{mean} curvature has the following scaling
\begin{equation*}
\text{width}(\theta)_{\rho=\rho_0,R}\sim t^{\frac{1}{2}}|K_0|^{-1/4}\sqrt{\sin\left(\frac{\pi}{n}\right)\sec(\theta)\sec\left(\frac{\pi}{n}-\theta\right)}.
\end{equation*}
\end{enumerate}
The overlap of these two regions forms the complete boundary layer in which the total bending energy is lowered by allowing localized stretching. The width of these boundary layers has the same scaling in thickness obtained by Lamb \cite{Lamb} in the context of vibrating shells and Efrati et. al. for non-Euclidean plates with $ K_0>0 $ \cite{Efrati2007Buckling}. The fact that there is no boundary layer to reduce the mean curvature in the case $ n=2 $ is a consequence of the fact that $ \eta=xy $ is a minimal surface in the FvK approximation. 

The width of the boundary layers near the lines of inflection scale like
\begin{equation*}
\text{width}(\rho)_{\eta=0}\sim t^{\frac{1}{3}}\rho^{\frac{1}{3}}|K_0|^{-\frac{1}{6}},
\end{equation*}
which has the same scaling with thickness, but not $\rho$, for minimal ridges formed by crumpling \cite{Lobkovsky1996}. In this boundary layer the mean curvature is locally reduced by correcting a jump discontinuity in the azimuthal curvature. This reduction of energy near this type of  singularity is different from the regularization near a ridge singularity in which the bending energy diverges while the stretching converges to zero with decreasing thickness \cite{Shankar:Ridge, Conti2008}. 

We conclude the paper in section 8 with a discussion of the implications of the results of this paper and future mathematical directions stemming from this work.
%

%
%
\section{F\"{o}ppl - von K\'arm\'an equations for metrics with constant negative Gaussian  curvature}

\subsection{F\"{o}ppl - von K\'arm\'an ansatz and elastic energy}
Let $\mathbb{H}^2_{K_0}$ denote the hyperbolic plane with Gaussian curvature $K_0$ and let $\mathcal{D}\subset\mathbb{H}_{K_0}$ be an annulus of inner radius $\rho_0$ and outer radius $R$ centered at the origin of $\mathbb{H}_{K_0}$. Introducing the \textbf{dimensionless curvature} $\epsilon$ defined by 
\begin{equation}\label{FvK:epsilon}
\epsilon=\sqrt{|K_0|}R
\end{equation} 
it follows that in geodesic polar coordinates $(\rho,\theta)$, the metric on $\mathcal{D}$ is \cite{SpivakII}: 
\begin{equation*}
\mathbf{g}=\rmd \rho^2+ \frac{\sinh^2(\epsilon\rho/R)}{|K_0|}\,\rmd \theta^2.
\end{equation*}
Therefore, expanding in $\epsilon$, $\mathbf{g}$ can be approximated to order $\epsilon^2$ by $\mathbf{g}=\mathbf{g}_{0}+\epsilon^2\mathbf{g}_1$ where $g_0$ is the Euclidean metric on $\mathbb{R}^2$ and
\begin{equation*}
\mathbf{g}_1=\frac{\rho^4}{3R^2}\,\rmd \theta^2=\frac{1}{3R^2}\left(v^2\,\rmd u^2-2uv\,\rmd u\rmd v+u^2\,\rmd v^2\right),
\end{equation*}
with the ``Cartesian'' coordinates $u=\rho\cos(\theta)$ and $v=\rho\sin(\theta)$. 

A deformation $\mathbf{x}:\mathcal{D}\rightarrow \mathbb{R}^3$ can be constructed that is an isometric immersion of $\mathbf{g}$ to order $\Or(\epsilon)$ by assuming that
\begin{equation}\label{small-slopes-configuration}
\mathbf{x}=i+\epsilon i_{\perp}\circ \eta +\epsilon^2 i\circ\chi,
\end{equation}
where $\chi\in W^{1,2}(\mathcal{D},\mathbb{R}^2)$ and $\eta\in W^{2,2}(\mathcal{D},\mathbb{R})$. In this context, when we write that a deformation $\mathbf{x}:\mathcal{D}\rightarrow \mathbb{R}^3$ satisfies $\mathbf{x}\in \mathcal{A}$ we mean that there exists $(\chi,\eta)\in \mathcal{A}$ such that $\mathbf{x}=i+\epsilon i_{\perp}\circ\eta+\epsilon^2 i\circ\chi$. Based on (\ref{intro:FvKenergy}) we define the elastic energy $ E_t:\mathcal{A}\rightarrow \mathbb{R} $ by 
\begin{eqnarray*}
\fl E_t[\mathbf{x}]=\frac{Y\epsilon^4}{8(1+\nu)}\int_{\mathcal{D}}\left[\left(\frac{1}{1-\nu}\tr(\gamma)^2-2\det(\gamma)\right)\right.\\
\left.+\frac{t^2}{3R^2\epsilon^2}\left(\frac{1}{1-\nu}\tr(D^2\eta)^2-2\det(D^2\eta)\right)\right]\,\rmd u \rmd v, \label{FvK:energy}.
\end{eqnarray*}
where again $\gamma$ is the in-plane strain defined by (\ref{FvK:strain}) and $D^2\eta$ denotes the Hessian matrix of second partial derivatives of $\eta$. A deformation $\mathbf{x}\in \mathcal{A}$ is called an \textbf{equilibrium configuration} in the FvK ansatz if
\begin{equation*}
E_t[\mathbf{x}]=\inf_{\mathbf{y}\in \mathcal{A}}E_t[\mathbf{y}].
\end{equation*}

Define the dimensionless variables $x,y, \chi^{\prime}$, and $\eta^{\prime}$ by
\begin{equation*}
\begin{array}{ccccc}
u=Rx, & v=Ry, & r=R\rho, & \chi=R\chi^{\prime},  & \eta=R\eta^{\prime},
\end{array}
\end{equation*} 
and $(D^2)^{\prime}$ the Hessian operator in the coordinates $x$ and $y$. Then $(D^2)^{\prime}\eta^{\prime}=D^2 \eta$ and thus introducing the \textbf{dimensionless thickness} $\tau$ defined by
\begin{equation}
\tau=t/\left(\sqrt{3}R\epsilon\right)
\end{equation}
and dropping the $\prime $ notation we have that
\begin{eqnarray}
\fl \frac{8(1+\nu)}{Y\epsilon^4R^2}E_t[\mathbf{x}]&=& \mathcal{E}_{\tau}[\mathbf{x}]
=\mathcal{S}[\mathbf{x}]+\tau^2\mathcal{B}[\mathbf{x}] \nonumber\\
&=& \int_{B}\left(\frac{1}{1-\nu}\text{tr}(\gamma)^2-2\det(\gamma)\right)\,\rmd x \rmd y \nonumber\\
&\,&+\tau^2\int_{B}\left(\frac{1}{1-\nu}\text{tr}\left(\left(D^2\right)^{\prime}\eta^{\prime}\right)-2\det\left(\left(D^2\right)^{\prime}\eta^{\prime}\right)\right)\,\rmd x \rmd y. \label{FvK:normalizedenergy}
\end{eqnarray}
where $B$ is an annulus with inner radius $r_0=\rho_0/R$ and unit outer radius. Minimizers of $E_t$ also minimize $\mathcal{E}_{\tau}$ and therefore we will restrict our attention to the functional $\mathcal{E}_{\tau}$.

%
%
\subsection{Euler-Lagrange equations}
Throughout the rest of the paper let $\partial B$ denote the boundary of $B$ with outward normal vector field $\mathbf{n}$ and tangent vector field $\mathbf{t}$. In this subsection we will assume that $\mathbf{x}\in \mathcal{A}$ extremizes $\mathcal{E}_{\tau}$ with out-of-plane displacement $\eta$ and in-plane displacement $\chi=(\chi_1,\chi_2)$. 

 Taking the variation of $\mathcal{E}_{\tau}$ with respect to $\chi_1$ we obtain:
 \begin{eqnarray*}
\fl \delta_{\chi_1}\mathcal{E}_{\tau}[\mathbf{x}]&=&\int_{B}\left[\frac{4}{1-\nu}\left(\gamma_{11}+\gamma_{22}\right)\frac{\partial}{\partial x} \delta \chi_1-4\left(\gamma_{22}\frac{\partial}{\partial x}\delta \chi_1-\gamma_{12}\frac{\partial}{\partial y} \delta \chi_1\right)\right]\,\rmd x \rmd y\\
 \fl &=& 4\int_{B}\left(\frac{\gamma_{11}+\nu \gamma_{22}}{1-\nu},\gamma_{12}\right)\cdot \nabla \delta \chi_1\,\rmd x \rmd y\\
 \fl &=& 4\int_{\partial B}\left(\frac{\gamma_{11}+\nu \gamma_{22}}{1-\nu}, \gamma_{12}\right)\cdot \mathbf{n} \delta \chi_1\,\rmd s-4\int_{B}\nabla \cdot \left(\frac{\gamma_{11}+\nu \gamma_{22}}{1-\nu}, \gamma_{12}\right)\delta \chi_1\,\rmd x \rmd y.
 \end{eqnarray*}
 Similarly, the variation with respect to $\chi_2$ yields:
 \begin{eqnarray*}
\fl \delta \chi_2 \mathcal{E}_\tau[\mathbf{x}] = 4\int_{\partial B}\left(\gamma_{12},\frac{\gamma_{22}+\nu \gamma_{11}}{1-\nu}, \right)\cdot \mathbf{n} \delta \chi_2\,\rmd s-4\int_{B}\nabla \cdot \left(\gamma_{12},\frac{\gamma_{22}+\nu \gamma_{11}}{1-\nu}, \right)\delta \chi_1\,\rmd x \rmd y.
 \end{eqnarray*}
 Therefore, the governing equations for the in-plane strain are
 \begin{equation}
 \begin{array}{cc}
 \nabla \cdot \left(\frac{\gamma_{11}+\nu\gamma_{22}}{1-\nu},\gamma_{12}\right)=0, & \nabla \cdot \left(\gamma_{12}, \frac{\nu \gamma_{11}+\gamma_{22}}{1-\nu}\right)=0, \label{FvK:StrainEqun}
 \end{array}
 \end{equation}
 with the natural boundary conditions 
 \begin{equation*}
 \begin{array}{cc}
 \left.\mathbf{n}\cdot\left(\frac{\gamma_{11}+\nu\gamma_{22}}{1-\nu},\gamma_{12}\right)\right|_{\partial B}=0, & 
 \left.\mathbf{n}\cdot  \left(\gamma_{12}, \frac{\nu\gamma_{11}+\gamma_{22}}{1-\nu}\right)\right|_{\partial B}=0.
 \end{array}
 \end{equation*}
Furthermore, (\ref{FvK:StrainEqun}) implies that there exists a potential function $\Phi\in W^{2,2}(B,\mathbb{R})$ satisfying 
\begin{equation}\label{FvK:stressfunction}
D^2\Phi=\frac{1}{1-\nu}\left(\begin{array}{cc}
\nu \gamma_{11}+\gamma_{22} & -(1-\nu)\gamma_{12}\\
-(1-\nu)\gamma_{12} & \gamma_{11} +\nu \gamma_{22}
\end{array}\right).
\end{equation}

Now, if we assume $\gamma$ is second differentiable, then differentiating and eliminating terms involving $\chi$ from (\ref{FvK:strain}) we have the following \emph{compatibility condition} between $\gamma$ and $\eta$:
\begin{equation}\label{FvK:comp}
\frac{\partial^2 \gamma_{11}}{\partial y^2}-2\frac{\partial^2 \gamma_{12}}{\partial x \partial y}+\frac{\partial^2 \gamma_{22}}{\partial x^2}-2[\eta,\eta]-2=0,
\end{equation}
where the operator $[\cdot,\cdot]:\mathcal{C}^2(B,\mathbb{R})\times \mathcal{C}^2(B,\mathbb{R})\rightarrow \mathbb{R}$ is defined by
\begin{equation*}
[f,g]=\frac{1}{2}\left(\frac{\partial^2 f}{\partial x^2}\frac{\partial g}{\partial y^2}+\frac{\partial^2 f}{\partial y^2} \frac{\partial^2 g}{\partial x^2}-2\frac{\partial^2 f}{\partial x \partial y}\frac{\partial^2 g}{\partial x \partial y}\right).
\end{equation*}
Therefore, inverting (\ref{FvK:stressfunction}) and substituting into (\ref{FvK:comp}) we have the following proposition.
\begin{proposition}
If $\mathbf{x}^*\in \mathcal{A}$ with potential $\Phi$ and out-of-plane displacement $\eta$, $\mathcal{E}_{\tau}[\mathbf{x}^*]=\inf_{\mathbf{x}\in \mathcal{A}}\mathcal{E}_{\tau}[\mathbf{x}]$, and if $\Phi$ is (weakly) four times differentiable then $\Phi$ satisfies the \textbf{first F\"oppl - von K\'arm\'an} equation:
\begin{equation}\label{FvK:firstFVK}
\frac{1}{2(1+\nu)}\Delta^2 \Phi+[\eta,\eta]=-1.
\end{equation}
Also, $\Phi$ satisfies the following \textbf{natural boundary conditions}:
\begin{equation}
\begin{array}{cc}
\left.\mathbf{n}\cdot\left(\frac{\partial^2 \Phi}{\partial y^2},-\frac{\partial^2 \Phi}{\partial x \partial y}\right)\right|_{\partial B}=0, & \left. \mathbf{n}\cdot \left(-\frac{\partial^2 \Phi}{\partial x \partial y},\frac{\partial^2 \Phi}{\partial x^2}\right)\right|_{\partial B}=0. \label{FvK:StressBC}
\end{array}
\end{equation}
\end{proposition}

Furthermore, by the definition of the potential $\Phi$ given in by (\ref{FvK:stressfunction}) the stretching energy can be expressed as:
\begin{equation*}
\mathcal{S}[\mathbf{x}]=\int_{B}\left(\frac{1}{1+\nu}\left(\Delta \Phi\right)^2-2[\Phi,\Phi]\right)\,\rmd x \rmd y.
\end{equation*}
The next proposition allows us to simplify this energy further and will be useful in the next section.

\begin{proposition}\label{FvK:stretchenergyprop}
If $\mathbf{x^*}\in \mathcal{A}$ with potential $\Phi$ and $\inf_{\mathbf{x}\in \mathcal{A}}\mathcal{E}_{\tau}[\mathbf{x}]=\mathcal{E}_{\tau}[\mathbf{x}^*]$ then
\begin{equation*}
\mathcal{S}[\mathbf{x}^*]=\int_{B}\frac{1}{1+\nu}\left(\Delta \Phi\right)^2\,\rmd x \rmd y.
\end{equation*}
\end{proposition}
\begin{proof}
Let $\mathbf{x^*}\in \mathcal{A}$ with corresponding potential function $\Phi$ such that $\inf_{\mathbf{x}\in \mathcal{A}}\mathcal{E}_{\tau}[\mathbf{x}]=\mathcal{E}_{\tau}[\mathbf{x}^*]$. Using Stokes' theorem and applying the boundary conditions (\ref{FvK:StressBC}), we have that
\begin{eqnarray*}
\fl \int_{B}[\Phi,\Phi]\,\rmd x \rmd y=\int_{B} \left(\frac{\partial^2\Phi}{\partial x^2}\frac{\partial^2 \Phi}{\partial y^2}-\frac{\partial^2 \Phi}{\partial x \partial y}\frac{\partial^2 \Phi}{\partial x \partial y}\right)\,dxdy\\
=\int_{B}\left[\frac{\partial}{\partial x} \left(\frac{\partial \Phi}{\partial x}\frac{\partial^2 \Phi}{\partial y^2}-\frac{\partial \Phi}{\partial y}\frac{\partial^2 \Phi}{\partial x \partial y}\right)+\frac{\partial}{\partial y}\left(\frac{\partial \Phi}{\partial y}\frac{\partial^2 \Phi}{\partial x^2}-\frac{\partial \Phi}{\partial y}\frac{\partial^2 \Phi}{\partial x \partial y}\right)\right]\,\rmd x \rmd y\\
=\int_{\partial B}\left(\frac{\partial \Phi}{\partial x} \frac{\partial^2 \Phi}{\partial y^2}-\frac{\partial \Phi}{\partial y}\frac{\partial^2 \Phi}{\partial x \partial y},\frac{\partial \Phi}{\partial y}\frac{\partial^2 \Phi}{\partial x^2}-\frac{\partial \Phi}{\partial x}\frac{\partial^2 \Phi}{\partial x \partial y}\right)\cdot\mathbf{n}\,\rmd s\\
= \int_{\partial B} \frac{\partial \Phi}{\partial x}\left(\frac{\partial^2 \Phi}{\partial y^2},-\frac{\partial^2 \Phi}{\partial x \partial y}\right)\cdot \mathbf{n}\,\rmd s+\int_{\partial B} \frac{\partial \Phi}{\partial y}\left(-\frac{\partial^2 \Phi}{\partial x \partial y},\frac{\partial^2 \Phi}{\partial x \partial y}\right)\cdot \mathbf{n}\,\rmd s\\
=0.
\end{eqnarray*}
Therefore,
\begin{equation*}
\mathcal{S}[\mathbf{x}^*]=\int_{B}\left(\frac{1}{1+\nu}\left(\Delta \Phi\right)^2-2[\Phi,\Phi]\right)\,\rmd x \rmd y=\int_{B}\frac{1}{1+\nu}\left(\Delta \Phi\right)^2\,\rmd x \rmd y.
\end{equation*}
\end{proof}

%
%
Now, taking the variation of $\mathcal{S}$ with respect to $\eta$ and applying the boundary conditions (\ref{FvK:StressBC}) we have
\begin{eqnarray*}
 \fl \delta_{\eta}\mathcal{S}[\mathbf{x}]&=&\int_{B}\left[\frac{1}{1-\nu}\left(\gamma_{11}+\gamma_{22}\right)\left(\frac{\partial \eta}{\partial x}\frac{\partial}{\partial x} \delta \eta +\frac{\partial \eta}{\partial y}\frac{\partial}{\partial y}\delta \eta\right)\right.\nonumber\\
\fl &\,&\left.- \left(\gamma_{11} \frac{\partial \eta}{\partial y}\frac{\partial}{\partial y} \delta \eta+\gamma_{22}\frac{\partial \eta}{\partial x}\frac{\partial}{\partial x}\delta \eta-\gamma_{12}\frac{\partial \eta}{\partial x}\frac{\partial}{\partial y} \delta \eta-\gamma_{12}\frac{\partial \eta}{\partial y}\frac{\partial}{\partial x}\delta \eta\right)\right]\,\rmd x \rmd y\\
\fl &=& 4\int_{B}\left(\frac{\partial^2 \Phi}{\partial y^2}\frac{\partial \eta}{\partial x}-\frac{\partial^2 \Phi}{\partial x \partial y}\frac{\partial \eta}{\partial y}, \frac{\partial^2 \Phi}{\partial x^2}\frac{\partial \eta}{\partial y}-\frac{\partial^2 \Phi}{\partial x \partial y}\frac{\partial \eta}{\partial x}\right)\cdot \nabla \delta \eta\,\rmd x \rmd y\\
\fl &=&\int_{\partial B}\left[\frac{\partial \eta}{\partial x}\left(\frac{\partial^2 \Phi}{\partial y^2},-\frac{\partial^2 \Phi}{\partial x \partial y}\right)\cdot \mathbf{n}+\frac{\partial \eta}{\partial y}\left(-\frac{\partial^2 \Phi}{\partial x \partial y},\frac{\partial^2 \Phi}{\partial x \partial y}\right)\cdot \mathbf{n}\right]\delta \eta\,\rmd s \nonumber\\
\fl &\,&-4\int_{B}[\Phi,\eta]\delta \eta\,\rmd x \rmd y\\
\fl &=& -4\int_{B}[\Phi,\eta]\delta \eta\,\rmd x \rmd y.
\end{eqnarray*}
Furthermore, following the derivation in \cite{Landau}, the variation of $\mathcal{B}$ with respect to $\eta$ is given by
\begin{eqnarray*}
 \fl \delta_{\eta}\mathcal{B}[\mathbf{x}]&=&\int_{B}\Delta^2 \eta\,\delta \eta\,\rmd x \rmd y+\int_{\partial B}\left[\frac{1}{1-\nu}\Delta \eta -\mathbf{n}^T\cdot D^2 \eta \cdot \mathbf{n}\right]\, \frac{\partial \delta\eta}{\partial \mathbf{n}}\,\rmd s\\
\fl &\,&-\int_{\partial B}\left[\frac{1}{1-\nu}\frac{\partial \Delta \eta}{\partial \mathbf{n}}+\frac{\partial}{\partial \mathbf{t}}\left(\mathbf{n}^T\cdot D^2 \eta \cdot \mathbf{t}\right)\right]\,\delta \eta\,\rmd s,
\end{eqnarray*}
where $\phi$ is the angle $\mathbf{n}$ makes with the $x$-axis. Therefore, we have the following proposition.

\begin{proposition}
If $\mathbf{x}^*\in \mathcal{A}$ with potential $\Phi$ and out-of-plane displacement $\eta$, $\mathcal{E}_{\tau}[\mathbf{x}^*]=\inf_{\mathbf{x}\in \mathcal{A}}\mathcal{E}_{\tau}[\mathbf{x}]$, and if $\eta$ is (weakly) four times differentiable then
\begin{equation}\label{FvK:secondFVK}
[\Phi,\eta]=\frac{\tau^2}{4(1-\nu)}\Delta^2 \eta.
\end{equation}
This equation is called the \textbf{second F\"oppl - von K\'arm\'an} equation. Furthermore, $\eta$ satisfies the following \textbf{natural boundary conditions}:
\begin{equation} \label{FvK:outofplaneBC1}
\left.\frac{1}{1-\nu}\Delta \eta -\mathbf{n}^T\cdot D^2 \eta \cdot \mathbf{n}\right|_{\partial B}=0, 
\end{equation}
\begin{equation} \label{FvK:outofplaneBC2}
\left.\frac{1}{1-\nu}\frac{\partial \Delta \eta}{\partial \mathbf{n}}+\frac{\partial}{\partial \mathbf{t}}\left(\mathbf{n}^T\cdot D^2 \eta \cdot \mathbf{t}\right)\right|_{\partial B}=0.
\end{equation}
\end{proposition}

\begin{remark}
The equations (\ref{FvK:firstFVK}) and (\ref{FvK:secondFVK}) along with the boundary conditions (\ref{FvK:StressBC}), (\ref{FvK:outofplaneBC1}) and (\ref{FvK:outofplaneBC2}) are a complete system of equations governing the deformation of the sheet. For general geometries, the boundary conditions are intractable and in later sections we will use the geometry of $B$ to simplify these equations further.
\end{remark}

\section{Flat solution}
Define the \textbf{admissible set of flat configurations} $\mathcal{A}_f\subset \mathcal{A}$ by $\mathcal{A}_f=\{\mathbf{x}\in \mathcal{A}:\eta=0\}$ and define $\mathcal{F}=\inf_{\mathbf{x}\in \mathcal{A}_f}\mathcal{E}_{\tau}[\mathbf{x}]$. If $\mathbf{x}\in \mathcal{A}_f$ extremizes $\mathcal{E}_{\tau}$ with potential function $\Phi$ then, assuming radial symmetry, we have by the first FvK equation (\ref{FvK:firstFVK}) and the boundary condition (\ref{FvK:StressBC}) that $\Phi$ satisfies the following boundary value problem
\begin{equation}\label{flat:bvp}
\frac{\partial^4 \Phi}{\partial r^4}+\frac{2}{r}\frac{\partial^3 \Phi}{\partial r^3}-\frac{1}{r^2}\frac{\partial^2 \Phi}{\partial r^2}+\frac{1}{r^3}\frac{\partial \Phi}{\partial r}=-2(1+\nu), \qquad \left.\frac{\partial \Phi}{\partial r}\right|_{r=r_0,1}=0.
\end{equation}

The general solution to this differential equation is
\begin{equation*}
\Phi=c_1+c_2r^2+c_3\ln(r)+c_4r^2\ln(r)-(1+\nu)\frac{r^4}{32},
\end{equation*}
where $c_1,c_2,c_3$ and $c_4$ are arbitrary constants. Since $c_1$ and $c_3\ln(r)$ are harmonic functions it follows from proposition \ref{FvK:stretchenergyprop} that we can assume without loss of generality that  $c_1=c_3=0$. Consequently, the solution to (\ref{flat:bvp}) is
\begin{equation*}
\Phi=\frac{1+\nu}{32\ln\left(r_0\right)}\left[\left(r_0^2-1-2\ln(r_0)\right)r^2+2\left(r_0^2-1\right)r^2\ln(r)-\ln\left(r_0\right)r^4\right].
\end{equation*} 
Therefore, we have the following result:
\begin{lemma} \label{flat:upperbound} Suppose $\mathbf{x}^*\in \mathcal{A}_f$ satisfies $\mathcal{E}_{\tau}[\mathbf{x}^*]=\inf_{\mathbf{x}\in \mathcal{A}_f}\mathcal{E}_{\tau}[\mathbf{x}]$. Then,
\begin{equation*} 
\fl \mathcal{E}_{\tau}[\mathbf{x}^*]=\mathcal{F}=\frac{(1+\nu)\pi}{32\ln(r_0)^2}\left[\frac{1}{2}\left(r_0^2-1\right)^3+\left(r_0^2-1\right)^2\left(r_0^2+1\right)\ln(r_0)+\frac{2}{3}\left(1-r_0^6\right)\ln(r_0)^2\right] .
\end{equation*}
\end{lemma}

%
%
\section{Isometric immersions}\label{isoimmersions}
Define the set of \textbf{isometric immersions} by $\mathcal{A}_0=\{\mathbf{x}\in \mathcal{A}: \gamma=0\}$. By the compatibility condition (\ref{FvK:comp}), it follows that the solvability condition for the equation $\gamma=0$ is the Monge-Ampere equation
\begin{equation}\label{isoimmersions:monge-ampere}
[\eta,\eta]=\det\left(D^2\eta\right)=-1,
\end{equation}
which is a version of Gauss's Theorema Egregium in the FvK ansatz. Consequently, if $\mathbf{x}\in \mathcal{A}_0$ with corresponding out-of-plane displacement $\eta$ and $\mathbf{x}$ satisfies (\ref{FvK:firstFVK}) and (\ref{FvK:secondFVK}) then $\Delta^2 \eta=0$. Therefore, it is natural to look for solutions to the FvK equations by finding solutions to the above Monge-Ampere equation (\ref{isoimmersions:monge-ampere}) satisfying $ \Delta^2\eta=0$.

\subsection{Saddle isometric immersions and convergence of minimizers}

By the definition of $\mathcal{E}_{\tau}$ (\ref{FvK:normalizedenergy})  and the Monge-Ampere equation (\ref{isoimmersions:monge-ampere}) it follows that for all $\mathbf{x}\in \mathcal{A}_0$
\begin{equation}\label{isoimmersions:energy}
\frac{1}{\tau^2}\inf_{\mathbf{x}\in \mathcal{A}_0}\mathcal{E}_{\tau}[\mathbf{x}]=\inf_{\mathbf{x}\in \mathcal{A}_0}\int_{B}\frac{1}{1-\nu}\left(\Delta \eta\right)^2\,\rmd x \rmd y+2\pi(1-r_0^2).
\end{equation}
Consequently, since $\eta=xy$ satisfies (\ref{isoimmersions:monge-ampere}) and is harmonic, it follows that the global minimum of (\ref{isoimmersions:energy}) over $\mathcal{A}_0$ is obtained. Therefore by the arguments presented in \cite{Gemmer2011} we have the following results concerning the convergence and scaling of minimizers of $\mathcal{E}_{\tau}$:
\begin{theorem}
Let $\mathbf{x}^*_{\tau}\in \mathcal{A}$ be a sequence with corresponding out-of-plane displacement $\eta^*_{\tau}$ such that $\inf_{\mathbf{x}\in \mathcal{A}}\mathcal{E}_{\tau}[\mathbf{x}]=\mathcal{E}_{\tau}[\mathbf{x}^*_{\tau}]$. Then,
\begin{enumerate}
\item $\displaystyle{
\lim_{\tau\rightarrow 0}\inf_{\mathbf{x}\in \mathcal{A}}\frac{1}{\tau^2}\mathcal{E}_{\tau}[\mathbf{x}]=\lim_{\tau\rightarrow 0}\frac{1}{\tau^2}\mathcal{E}_{\tau}[\mathbf{x^*_{\tau}}]=\inf_{\mathbf{x}\in \mathcal{A}_0}\mathcal{B}[\mathbf{x}]=2\pi(1-r_0^2).} $
\item There exists a subsequence $ \eta^*_{\tau_k}$ and $ \mathbf{x}^*\in \mathcal{A}_0 $ with corresponding out-of-plane displacement $ \eta^* $ such that $\eta^*_{\tau_k}\rightharpoonup \eta^*$. Moreover, there exists $ A\in SO(2)$ and $ b\in \mathbb{R}$ such that $ \eta^*(A(x,y))+b=xy$.
\end{enumerate} 
\end{theorem}
\begin{proposition} \label{isoimmersion:upperbound}  Let $\mathbf{x}^*\in \mathcal{A}$ such that $\mathcal{E}_{\tau}[\mathbf{x}^*]=\inf_{\mathbf{x}\in \mathcal{A}}\mathcal{E}_{\tau}[\mathbf{x}]$. Then, 
\begin{equation*}
\mathcal{E}_{\tau}[\mathbf{x}^*]\leq \min\{2\pi\tau^2 (1-r_0^2),\mathcal{F}\}.
\end{equation*}
\end{proposition}
\begin{remark}
There are two natural test functions for the elastic energy, flat deformations and isometric immersions, the upper bounds in the above lemma correspond to the minimum of the elastic energy over these two classes of deformations.
\end{remark}

\subsection{Periodic isometric immersions}
Let
\begin{equation}\label{isoimmersions:periodicdisplacement}
\overline{\eta}_n=y\left(x-\cot\left(\frac{\pi}{n}\right)y\right)=\frac{r^2}{2}\csc\left(\frac{\pi}{n}\right)\left[\cos\left(\frac{\pi}{n}-2\theta\right)-\cos\left(\frac{\pi}{n}\right)\right],
\end{equation}
which satisfiesthe Monge-Ampere equation (\ref{isoimmersions:monge-ampere}), $\Delta^2\overline{\eta}=0$ and $\overline{\eta}_n=0$ along the lines $\theta=0$ and $\theta=\pi/n$. Consequently, by taking odd periodic extensions of $\overline{\eta}_n$ (see figure \ref{fig:reflection}) an $n$-periodic isometric immersion $\eta_n$ can be constructed \cite{Gemmer2011}. Therefore, $\mathcal{A}_n\cap \mathcal{A}_0\neq \emptyset$ and by calculating the energy of this configuration we have the following result which agrees with proposition \ref{isoimmersion:upperbound} when $n=2$.
\begin{figure}[ht]
\begin{center}
\includegraphics[width=\textwidth]{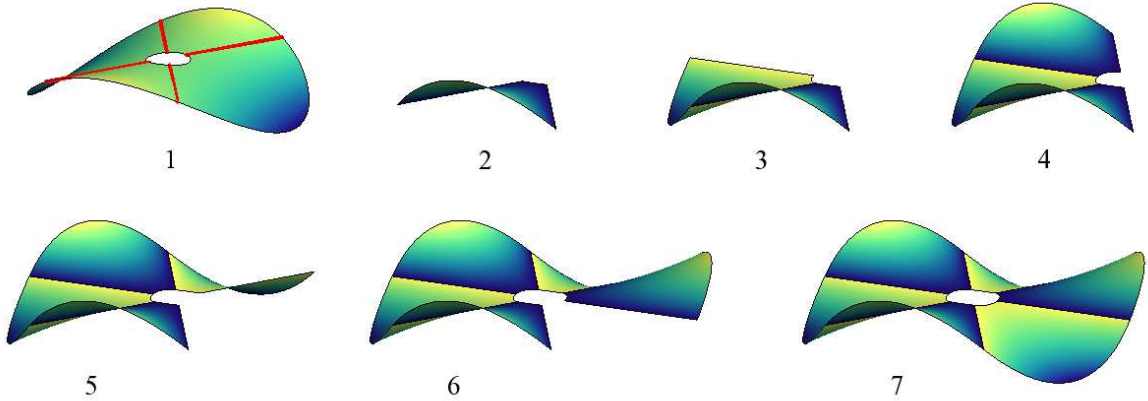}
\caption{\textbf{1.} The one parameter family of isometric immersions with out-of-plane displacement $\eta=y\left(x-\cot\left(\pi/n\right)\right)$ parametrized by $n\in\{2,3,\ldots\}$  vanishes along the lines $\theta=0$ and $\theta=\pi/n$. \textbf{2.} We can ``cut out" the section of the surface bounded between these two lines. \textbf{3.} We can take the odd reflection of this isolated piece of the surface about the the line $\theta=0$. \textbf{4-7.} By continually taking the odd reflection of sectors about lines where the surface vanishes we can construct a periodic isometric immersion.}
\label{fig:reflection}
\end{center}
\end{figure}
\begin{lemma} \label{isoimmersion:periodicupperbound} Let $\mathbf{x}^*\in \mathcal{A}_n$ such that $\mathcal{E}_{\tau}[\mathbf{x}^*]=\inf_{\mathbf{x}\in \mathcal{A}_n}\mathcal{E}_{\tau}[\mathbf{x}]$, then 
\begin{equation*}
\mathcal{E}_{\tau}[\mathbf{x}^*]\leq \min\left\{\tau^2\left(\frac{4\pi\cot^2(\pi/n)}{1-\nu}+2\pi\right)(1-r_0^2),\mathcal{F}\right\}\leq\min\left \{ Cn^2\tau^2,\mathcal{F}\right\},
\end{equation*}
where $C$ is a constant independent of $ n $ and $ \tau $.
\end{lemma}

 \begin{proof}
 The middle term in the chain of inequalities follows from calculating the bending energy of the deformation $\mathbf{x}_n$.  Expanding near $n=\infty$ we have that
 \begin{equation*}
  \frac{4\pi\cot^2\left(\frac{\pi}{n}\right)}{1-\nu}+2\pi=\frac{4n^2}{\pi(1-\nu)}+\left(2\pi -\frac{8\pi}{3(1-nu)}\right)+\ldots\,.
 \end{equation*}
 Consequently there exists $M>0$ and $K_1>0$ such that if $n\geq M$ then
 \begin{equation*}
 \frac{4\pi\cot^2\left(\frac{\pi}{n}\right)}{1-\nu}+2\pi\leq K_1n^2.
 \end{equation*}
 Differentiating we have that
  \begin{equation*}
  \frac{d}{dn}\cot^2\left(\frac{\pi}{n}\right)=\frac{2\cot^2\left(\frac{\pi}{n}\right)\csc\left(\frac{\pi}{n}\right)}{n^2}
  \end{equation*}
  and thus $\displaystyle{\cot^2\left(\frac{\pi}{n}\right)}$ is monotone increasing. Consequently, on the interval $[2,M]$ we have that
  \begin{equation*}
  \frac{4\pi\cot^2\left(\frac{\pi}{n}\right)}{1-\nu}+2\pi\leq \left(\frac{\pi \cot^2\left(\frac{\pi}{M}\right)}{1-\nu}+\frac{\pi}{2}\right)n^2=K_2n^2.
  \end{equation*}
  Therefore if we set $C=\max\{K_1,K_2\}$ then the result follows. 
 \end{proof}

\section{Numerical solutions}
The boundary value problem given by the FvK equations and its natural boundary conditions (\ref{FvK:firstFVK}-\ref{FvK:outofplaneBC2}) can be approximately solved by numerically minimizing $ \mathcal{E}_{\tau}$. To do this,  it is convenient to write the energy in the form
\begin{eqnarray} \label{FvK:numericalenergy}
\fl \mathcal{E}_{\tau}[\mathbf{x}]=\int_{B}\left[\frac{\nu}{1-\nu}\left(\gamma_{11}^2+\gamma_{22}
\right)^2+\gamma_{11}+2\gamma_{12}^2+\gamma_{22}^2\right]\,dxdy \nonumber\\
+\tau^2\int_{B}\left[\frac{\nu}{1-\nu}\left(\Delta \eta\right)^2+|D^2\eta|^2\right]\,dxdy.
\end{eqnarray}
$ \mathcal{E}_{\tau}$ can then be discretized using a finite difference scheme to approximate the integrand and a quadrature rule to approximate the integrals. This discretization of (\ref{FvK:numericalenergy}) generates a sum of quadratic terms that can then be minimized using Matlab's minimization routine lsqnonlin \cite{Matlab:2010}.

In figure \ref{fig:convergencetosaddle} we plot the elastic energy of numerical minimizers of $\mathcal{E}_{\tau}$ for decreasing values of $ \tau$ with $ r_0=1$ and $ \nu=1/2$. The data in this figure was generated using the algorithm outlined in the previous paragraph with a $40\times40$ mesh in the coordinates $ r$ and $ \theta $, a fourth order approximation of derivatives, and a second order approximation of integrals. The solid line corresponds to the upper bound in lemma \ref{flat:upperbound} for flat deformations while the thin dotted line is the upper bound in proposition \ref{isoimmersion:upperbound} for isometric immersions. In this figure three representative minimizers are plotted and coloured by the Gaussian curvature $K=\det(D^2\eta)$ for various values of $ \tau $. These three surfaces and the scaling of the energy illustrate that with decreasing thickness the minimizing surface transitions from being flat to one that is close to the isometric immersion $\eta=xy$.  For the left most surface, the regions in which the Gaussian curvature is substantially different from $-1$ are localized to the edges of the annulus and shrink with decreasing thickness. This indicates that with decreasing thickness the stretching energy is being concentrated in boundary layers in which the bending energy of the isometric energy is slightly reduced.

\begin{figure}[ht]
\includegraphics[width=\textwidth]{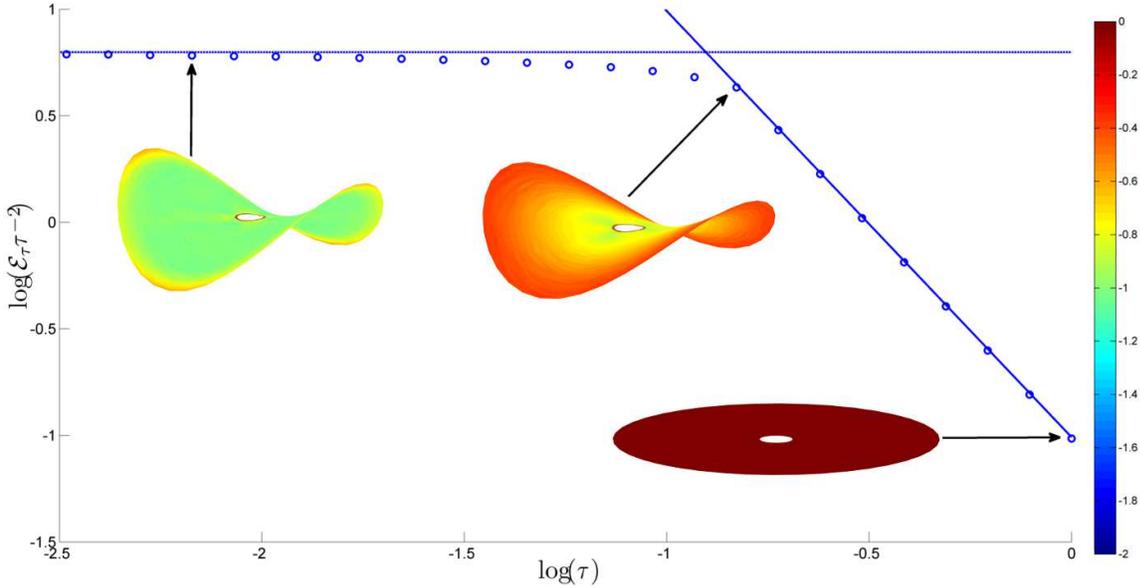}
\caption{The normalized elastic energy of numerical minimizers of $\mathcal{E}_{\tau}$ with $\nu=1/2$ and $ r_0=10^{-1}$ for decreasing values of $\tau$. The solid line corresponds to the elastic energy of minimizers over the set of flat deformations given by lemma \ref{flat:upperbound}. The dotted horizonal line corresponds to the global minimum of the elastic energy over the set of deformations satisfying $ \gamma=0 $.  The three configurations plotted are colored by the approximate Gaussian curvature $K=[\eta,\eta]$. These three surfaces illustrate that with decreasing thickness the surface transitions from being a flat surface to one that is close to the isometric immersion $\eta=xy$ with localized regions of stretching near the inner and outer radius of the disk.}
 \label{fig:convergencetosaddle}
\end{figure}

In figure \ref{fig:periodicconvergence} we again plot the numerical minimizers of $ \mathcal{E}_{\tau}$ using the same parameters and discretization used to generate figure \ref{fig:convergencetosaddle} but with the boundary conditions $ \eta=0 $ along the lines $ \theta=0,\pi/n, 2\pi/n,\ldots$ for $ n\in\{2,3,4,5\}$. These boundary conditions were selected to generate numerical minimizers over $ \mathcal{A}_n$. The dotted horizontal lines correspond to the upper bounds for $ n$-periodic isometric immersions given by lemma \ref{isoimmersion:periodicupperbound} while the solid line is again the upper bound in lemma \ref{flat:upperbound}. The four plotted surfaces are coloured by the approximate Gaussian curvature $K=[\eta,\eta]$ and were selected to compare the geometry of the boundary layers for $ n\geq 3$ with the case $ n=2 $. For $ n\geq 3 $ additional boundary layers form around the lines of inflection in which the surface stretches to reduce the local mean curvature of the surface.

\begin{figure}[ht]
\includegraphics[width=\textwidth]{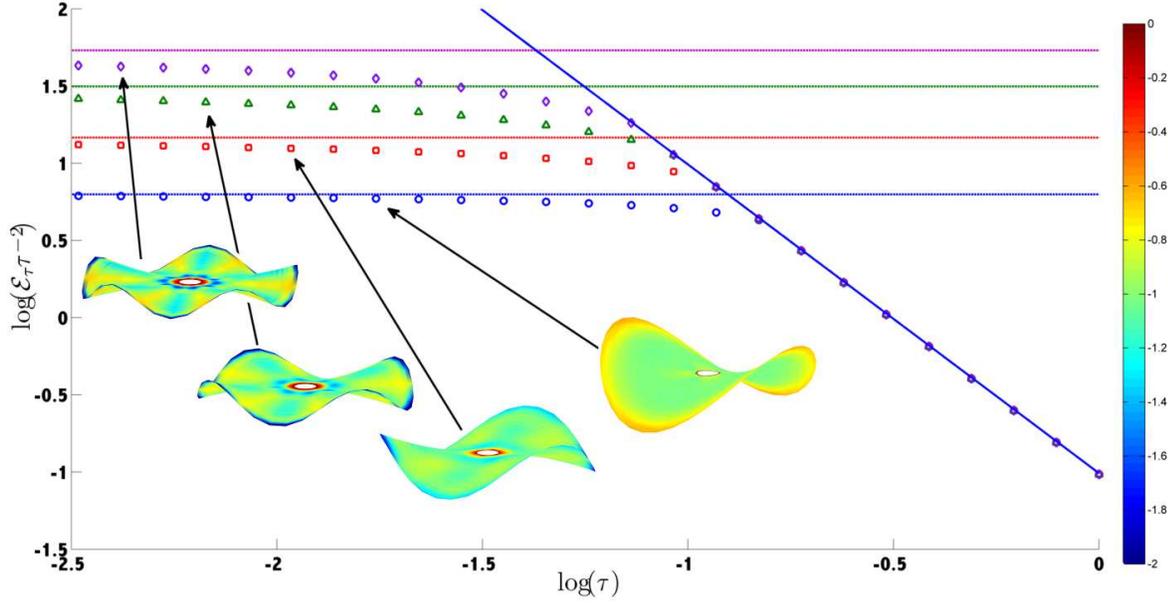}
\caption{The normalized elastic energy of numerical minimizers of $\mathcal{E}_{\tau}$ over $\mathcal{A}_n$ with $r_0=10^{-1}$ and $\nu=1/2$. The dotted horizontal lines correspond to the upper bounds given by lemma \ref{isoimmersion:periodicupperbound} while the solid line corresponds to the upper bound in lemma \ref{flat:upperbound}. The  configurations plotted are colored by the approximate Gaussian curvature $[\eta,\eta]$. In the vanishing thickness limit minimizers of $\mathcal{E}_{\tau}$ converge to an $n$-periodic isometric immersion. The numerical minimizers have localized regions of stretching near the inner and outer radius and along the lines of inflection. The existence of these regions indicate that the minimizers are perturbations of an isometric immersion with boundary layers to account for the natural boundary conditions.}
\label{fig:periodicconvergence}
\end{figure}

%
%
\section{Scaling laws for the elastic energy of periodic configurations}
In this section we derive ansatz free lower bounds for $n$-periodic configurations. The essential idea of this section is that the bending energy near the edge of the annulus controls the stretching energy in the bulk of the domain. Before we state and prove this lower bound we need several intermediate results. Let $n\in\{2,3,\ldots\}$ and define the following sets:
\begin{enumerate}
\item $S_n$ is the sector in $\mathbb{R}^2$ bounded between $\theta=0$ and $\theta=\pi/n$,
\item For $z>r_0$, $B^{z}\subset B$ is the annular region defined by $B^z=\{(r,\theta)\in B: r_0\leq z \leq r \leq 1 \}$,
\item $B^z_n=S_n\cap B^z$ is a wedge shaped region bounded between $\theta=0$, $\theta=\pi/n$, $r=1$, and $r=z$. 
\end{enumerate}
Furthermore, let $r^*=\max\{.95,r_0\}$. 
\begin{lemma} \label{perconfig:first-der-bound}
If $\mathbf{x}\in \mathcal{A}_n$ with corresponding out-of-plane displacement $\eta$ then
\begin{enumerate}
\item \begin{equation} 
\displaystyle{\int_{B_n^z}\left(\frac{\partial \eta}{\partial r}\right)^2 \,\rmd A\leq \frac{1}{n^2}\int_{B_n^z}\left(\frac{\partial^2 \eta}{\partial r \partial \theta}\right)^2\,\rmd A}
\end{equation}
\item 
\begin{equation}
\displaystyle{\int_{B_n^z}\eta^2\,\rmd A \leq \frac{\pi^4}{48n^4}\int_{B_n^z}\left(\frac{\partial^2 \eta}{\partial \theta^2}\right)^2\,\rmd A}
\end{equation}
\item 
\begin{equation}
\displaystyle{\int_{B_n^z}\left(\frac{\partial \eta}{\partial \theta}\right)^2\,\rmd A\leq \frac{\pi^2}{12n^2}\int_{B_n^z}\left(\frac{\partial^2 \eta}{\partial \theta^2}\right)^2\rmd A}
\end{equation}
\end{enumerate}
\end{lemma}

\begin{proof}
Since $\mathbf{x}\in \mathcal{A}_n$ we have that $\eta=0$ along the lines $\theta=0,\pi/n$ and thus
\begin{equation*}
\left.\frac{\partial \eta}{\partial r}\right|_{\theta=0,\frac{\pi}{n}}=0.
\end{equation*} 
Consequently it follows from Poincare's inequality with the optimal constant \cite{poincare-constant} that
\begin{equation*}
\int_0^{\frac{\pi}{n}}\left(\frac{\partial \eta}{\partial r}\right)^2\rmd \theta \leq \frac{1}{n^2}\int_0^{\frac{\pi}{n}}\left(\frac{\partial^2 \eta}{\partial r \partial \theta}\right)^2\rmd \theta.
\end{equation*}
Therefore, integrating we have that
\begin{equation*}
\int_{B_n^z}\left(\frac{\partial \eta}{\partial r}\right)^2\, \rmd A \leq \frac{1}{n^2}\int_{B_n^z}\left(\frac{\partial^2 \eta}{\partial r \partial \theta}\right)^2\rmd A, 
\end{equation*}
proving item (i).

Constructing the Green's function $G(\theta, \phi)$ for the operator $\frac{\partial^2 }{\partial \theta^2}$ with Dirichlet boundary conditions on $[0,\pi/n]$ we have that
\begin{equation*}
\eta(r,\theta)=\int_0^{\frac{\pi}{n}}G(\theta,\phi)\frac{\partial^2 \eta}{\partial \phi^2}\rmd \phi,
\end{equation*}
where 
\begin{equation*}
G(\theta,\phi)= \left\{
\begin{array}{ll}
\frac{n}{\pi}\phi(\theta-\pi/n), & \text{ if }\theta<\phi\\
\frac{n}{\pi}\theta(\phi-\pi/n), & \text{ if }\phi<\theta
\end{array} \right. .
\end{equation*}
Therefore,
\begin{eqnarray*}
\fl \eta^2(r,\theta)\leq \left(\int_0^{\frac{\pi}{n}}G^2(\theta,\phi)\rmd A \right) \left(\int_0^{\frac{\pi}{n}}\left(\frac{\partial^2 \eta}{\partial \theta^2}\right)^2\rmd\theta \right)= \frac{\theta^2(\pi-n\theta)^2}{3n\pi}\int_0^{\frac{\pi}{n}}\left(\frac{\partial^2 \eta}{\partial \theta^2}\right)^2\rmd\theta\\
 \leq \frac{\pi^3}{48n^3}\int_0^{\frac{\pi}{n}}\left(\frac{\partial^2 \eta}{\partial \theta^2}\right)^2\rmd\theta.
\end{eqnarray*}
Integrating we have that
\begin{equation*}
\int_{B_n^z}\eta^2\rmd A\leq \frac{\pi^4}{48 n^4}\int_{B_n^z}\left(\frac{\partial^2 \eta}{\partial \theta^2}\right)^2\rmd A,
\end{equation*}
proving item (ii).

Differentiating, we have that
\begin{equation*}
\frac{\partial \eta}{\partial \theta}=\int_{0}^{\frac{\pi}{n}}\frac{\partial G}{\partial \theta}\frac{\partial^2 \eta}{\partial \phi^2}\rmd \phi.
\end{equation*}
Therefore,
\begin{eqnarray*}
\fl \left(\frac{\partial \eta}{\partial \theta}\right)^2\leq \left(\int_0^{\frac{\pi}{n}} \left(\frac{\partial G}{\partial \theta}\right)^2\rmd \phi \right)\left(\int_0^{\frac{\pi}{n}}\left(\frac{\partial^2 \eta}{\partial \theta^2}\right)^2\rmd A\right)= \left(\frac{\pi}{3n}-\theta +\frac{n\theta^2}{\pi}\right)\int_0^{\frac{\pi}{n}}\left(\frac{\partial^2 \eta}{\partial \phi^2}\right)^2\rmd \phi\\
\leq  \frac{\pi}{12n}\int_0^{\frac{\pi}{n}}\left(\frac{\partial^2 \eta}{\partial \phi}\right)^2\rmd \phi.
\end{eqnarray*}
Integrating we have that
\begin{equation*}
\int_{B_n^z}\left(\frac{\partial \eta}{\partial \theta}\right)^2\rmd A\leq \frac{\pi^2}{12n^2}\int_{B_n^z}\left(\frac{\partial^2 \eta}{\partial \theta^2}\right)^2\rmd A,
\end{equation*}
proving item (iii).
\end{proof}

\begin{lemma} \label{perconfig:bound_second-der}
Let $\mathbf{x}\in \mathcal{A}_n$ with corresponding out-of-plane displacement $\eta$. If $\mathcal{B}[\mathbf{x}]\leq \mathcal{B}_0$ and $n>2$ then there exists a constant $C$ independent of $n$ and $\mathbf{x}$ such that
\begin{equation*}
\int_{B^{r^*}_n}\left[\left(\frac{\partial^2 \eta}{\partial \theta^2}\right)^2+\left(\frac{\partial^2 \eta}{\partial r \partial \theta}\right)^2\right]\rmd A\leq C\frac{\mathcal{B}_0}{n}.
\end{equation*}
\end{lemma}
\begin{proof}
Since $\mathcal{B}[\mathbf{x}]<\mathcal{B}_0$ it follows that
\begin{equation*}
\fl \int_{B_n^{r_0}}|D^2 \eta|^2\rmd A =\int_{B^{r_0}_n}\left[\frac{1}{r^2}\left(\frac{1}{r}\frac{\partial \eta}{\partial \theta}-\frac{\partial^2 \eta}{\partial r \partial \theta}\right)^2+\frac{1}{r^4}\left(\frac{\partial^2 \eta}{\partial \theta^2}+r\frac{\partial \eta}{\partial r}\right)^2+\left(\frac{\partial \eta}{\partial r}\right)^2\right]\rmd A\leq \frac{\mathcal{B}_0}{2n}.
\end{equation*}
Therefore,
\begin{equation*}
\int_{B^{r^*}_n}\left[\left(\frac{1}{r}\frac{\partial \eta}{\partial \theta}-\frac{\partial^2 \eta}{\partial r \partial \theta}\right)^2+\left(\frac{\partial^2 \eta}{\partial \theta^2}+r\frac{\partial \eta}{\partial r}\right)^2\right]\rmd A \leq \frac{\mathcal{B}_0}{2n}
\end{equation*}
and thus applying the elementary inequality $(a+b)^2\geq \frac{1}{2}a^2-2b^2$ we have that
\begin{equation*}
\fl \frac{1}{2}\int_{B^{r^*}_n}\left[\left(\frac{\partial^2 \eta}{\partial \theta^2}\right)^2+\left(\frac{\partial^2 \eta}{\partial r \partial \theta}\right)^2\right]\rmd A -2\int_{B^{r^*}_n}\left[\frac{1}{r^2}\left(\frac{\partial \eta}{\partial \theta}\right)^2+r^2\left(\frac{\partial \eta}{\partial r}\right)^2\right]\rmd A \leq \frac{\mathcal{B}_0}{2n}. 
\end{equation*}
Consequently, applying lemma \ref{perconfig:first-der-bound} we have that
\begin{eqnarray*}
\fl \int_{B^{r^*}_n}\left[\left(\frac{\partial^2 \eta}{\partial \theta^2}\right)^2+\left(\frac{\partial^2 \eta}{\partial r \partial \theta}\right)^2\right]\rmd A \leq \frac{\mathcal{B}_0}{n}+4\int_{B^{r^*}_n}\left[\frac{1}{(r^*)^2}\left(\frac{\partial \eta}{\partial \theta}\right)^2+\left(\frac{\partial \eta}{\partial r}\right)^2\right]\rmd A\\
\fl \qquad \qquad \leq \frac{\mathcal{B}_0}{n}+\frac{\pi^2}{3n^2(r^*)^2}\int_{B^{r^*}_n}\left(\frac{\partial^2 \eta}{\partial \theta^2}\right)^2\rmd A+\frac{4}{n^2}\int_{B^{r^*}_n}\left(\frac{\partial^2 \eta}{\partial r \partial \theta}\right)^2\rmd A.
\end{eqnarray*}
Therefore, since $r^*>\pi/(2\sqrt{3})$ and $n>2$ the result follows.
\end{proof}

\begin{lemma} \label{perconfig:stretch-upperbound}
Let $\mathbf{x}\in \mathcal{A}_n$ with corresponding out-of-plane displacement $\eta$. If $\mathcal{B}[\mathbf{x}]\leq \mathcal{B}_0$ and $n>2$ then there exists a constant $C$ independent of $\mathbf{x}$ and $n$ such that
\begin{equation*}
\int_{B^{r^*}}|\nabla \eta|^4\rmd A\leq C\frac{\mathcal{B}_0}{n^2}.
\end{equation*}
\end{lemma}
\begin{proof}
By lemmas \ref{perconfig:first-der-bound} and \ref{perconfig:bound_second-der} it follows that
\begin{eqnarray*}
\int_{B_n^{r^*}}\eta^2\,\rmd A &\leq& \frac{\pi^4}{48 n^4}\int_{B^{r^*}_n}\left(\frac{\partial^2 \eta }{\partial \theta^2}\right)^2\rmd A\leq C_1 \frac{\mathcal{B}_0}{n^5},
\end{eqnarray*}
\begin{eqnarray*}
\fl \int_{B_n^{r^*}}|\nabla \eta|^2\rmd A =\int_{B_n^{r^*}}\left[\left(\frac{\partial^2 \eta}{\partial r}\right)^2+\frac{1}{r^2}\left(\frac{\partial \eta}{\partial \theta}\right)^2\right]\rmd A \leq \int_{B_n^{r^*}}\left[\left(\frac{\partial^2 \eta}{\partial r}\right)^2+\frac{1}{(r^*)^2}\left(\frac{\partial \eta}{\partial \theta}\right)^2\right]\rmd A\\
\leq C_2\frac{\mathcal{B}_0}{n^3}.
\end{eqnarray*}
By rotational symmetry of $\mathcal{A}_n$ we have that
\begin{equation*}
\int_{B^{r^*}}\eta^2\,\rmd A\leq C_1 \frac{\mathcal{B}_0}{n^4} \text{ and } \int_{B^{r^*}}|\nabla \eta|^2\rmd A \leq C_2 \frac{\mathcal{B}_0}{n^2}
\end{equation*}
and therefore by a multiplicative inequality \cite{maz2011sobolev} it follows that there exists a constant $C_3$ independent of $\mathbf{x}$ and $n$ such that
\begin{eqnarray*}
\int_{B^{r^*}}|\nabla \eta|^4\rmd A &\leq& C_3\left(\int_{B^{r^*}}\eta^2 \rmd A\right)^{\frac{1}{2}}\left(\int_{B^{r^*}}\left(\eta^2+|\nabla \eta|^2+|D^2\eta|^2\right)\rmd A\right)^{\frac{3}{2}}\\
&\leq & C\frac{\mathcal{B}_0^{\frac{1}{2}}}{n^2}\left( C_1\frac{\mathcal{B}_0}{n^4}+C_2\frac{\mathcal{B}_0}{n^2}+\mathcal{B}_0\right)^{\frac{3}{2}}\\
&\leq & C\frac{\mathcal{B}_0^2}{n^2}.
\end{eqnarray*}
\end{proof}

\begin{lemma}\label{perconfig:stretch-upperbound-n2}
Let $\mathbf{x}\in \mathcal{A}_2$ with corresponding out-of-plane displacement $\eta$. If  $\mathcal{B}_{\tau}[\mathbf{x}]\leq \mathcal{B}_0$ then there exists a constant $C$ independent of $\mathbf{x}$ such that 
\begin{equation*}
\int_{B^{r^*}}|\nabla \eta|^4\rmd A\leq C\frac{\mathcal{B}_0^2}{4}
\end{equation*}
\end{lemma}
\begin{proof}
Since $\mathbf{x}\in \mathcal{A}_2$ it follows that the means of the functions $\eta$, $\frac{\partial \eta}{\partial x}$, and $\frac{\partial \eta}{\partial y}$ satisfy
\begin{equation*}
\begin{array}{ccc}
\fl \displaystyle{\frac{1}{\pi(1-(r^*)^2)}\int_{B^{r^*}}\eta\, \rmd A=0}, & \displaystyle{\frac{1}{\pi(1-(r^*)^2)}\int_{B^{r^*}}\frac{\partial \eta}{\partial x}\, \rmd A=0}, & \displaystyle{\frac{1}{\pi(1-(r^*)^2)}\int_{B^{r^*}}\frac{\partial \eta}{\partial y}\, \rmd A=0}.
\end{array}
\end{equation*}
Therefore, by Poincare's inequality it follows that there exists constants $C_1$ and $C_2$ independent of $\mathbf{x}$ such that
\begin{equation*}
\fl \int_{B^{r^*}}\left(\frac{\partial \eta}{\partial x}\right)^2\rmd A\leq C_1\int_{B^{r^*}}|\nabla \eta|^2\rmd A \text{ and } \int_{B^{r^*}}|\nabla \eta|^2\rmd A \leq C_2 \int_{B^{r^*}}|D^2\eta|^2\rmd A.
\end{equation*}
Consequently, by applying the same multiplicative inequality as in the proof of lemma \ref{perconfig:stretch-upperbound} the result follows.
\end{proof}

\begin{lemma}\label{perconfig:lower-bound-stretch}
Let $\mathbf{x}\in \mathcal{A}_n$. If  $\mathcal{B}[\mathbf{x}]\leq \mathcal{B}_0$ then there exists a constant $C>0$ independent of $\mathbf{x}$ and $n$ such that
\begin{equation*}
\mathcal{S}[\mathbf{x}]\geq \frac{1}{2}\mathcal{F}-C \frac{\mathcal{E}_0^2}{n^2}.
\end{equation*}
\end{lemma}

\begin{proof}  Let $\eta$ be the out-of-plane displacement corresponding to $\mathbf{x}$ and $\chi=(f,g)$ the in-plane displacement with component functions $f,g\in W^{1,2}(\mathbb{B},\mathbb{R})$. Then, by lemmas \ref{perconfig:stretch-upperbound} and \ref{perconfig:stretch-upperbound-n2} it follows that there exists a constant $C$ independent of $\mathbf{x}$ and $n$ such that
\begin{equation*}
\int_{B^{r^*}}|\nabla \eta|^4\rmd A \leq \int_{B}|\nabla \eta|^4\rmd A \leq C\frac{\mathcal{B}_0}{n^2}.
\end{equation*}
Therefore, the following inequalities hold:
\begin{enumerate}
\item \begin{eqnarray*}
\fl \int_{B^{r^*}}\gamma_{11}^2\,\rmd x \rmd y=\int_{B^{r^*}}\left(2\frac{\partial f}{\partial x}+\left(\frac{\partial \eta}{\partial x}\right)^2-\frac{y^2}{3}\right)\,\rmd x \rmd y \nonumber\\
\geq\frac{1}{2}\int_{B^{r^*}}\left(2\frac{\partial f}{\partial x}-\frac{y^2}{3}\right)^2\,\rmd x \rmd y-2\int_{B^{r^*}}\left(\frac{\partial \eta}{\partial x}\right)^4\,\rmd x \rmd y \nonumber\\
\geq \frac{1}{2}\int_{B^{r^*}}\left(2\frac{\partial f}{\partial x}-\frac{y^2}{3}\right)^2\,\rmd x \rmd y-2C\frac{\mathcal{B}_0^2}{n^2}.
\end{eqnarray*}
\item \begin{equation*}
\int_{B^{r^*}}\gamma_{22}^2\,\rmd x \rmd y\geq \frac{1}{2}\int_{B^{r^*}}\left(2\frac{\partial g}{\partial y}-\frac{x^2}{3}\right)^2\,\rmd x \rmd y-2C\frac{\mathcal{B}_0^2}{n^2}.
\end{equation*}
\item \begin{eqnarray*}
\fl \int_{B^{r^*}} \gamma_{12}^2\,\rmd x \rmd y\geq \frac{1}{2}\int_{B^{r^*}}\left(\frac{\partial f}{\partial y}+\frac{\partial g}{\partial x}+\frac{xy}{3}\right)^2\,\rmd x \rmd y-2\int_{B^{r^*}}\left(\frac{\partial \eta}{\partial x}\right)^2\left(\frac{\partial \eta}{\partial y}\right)^2\,\rmd x \rmd y\\
\geq \frac{1}{2}\int_{B^{r^*}}\left(\frac{\partial f}{\partial y}+\frac{\partial g}{\partial x}+\frac{xy}{3}\right)^2\,\rmd x \rmd y\\
-2\left(\int_{B^{r^*}}\left(\frac{\partial \eta}{\partial x}\right)^4\,\rmd x \rmd y\right)^{1/2}\left(\int_{B^{r^*}}\left(\frac{\partial \eta}{\partial y}\right)^4\,\rmd x \rmd y\right)^{1/2}\\
\geq  \frac{1}{2}\int_{B^{r^*}}\left(\frac{\partial f}{\partial y}+\frac{\partial g}{\partial x}+\frac{xy}{3}\right)^2\,\rmd x \rmd y-2C\frac{\mathcal{B}_0^2}{n^2}.
\end{eqnarray*}
\item \begin{eqnarray*}
\fl \int_{B^{r^*}}(\gamma_{11}+\gamma_{22})^2\,\rmd x \rmd y\geq \frac{1}{2}\int_{B^{r^*}} \left(2\frac{\partial f}{\partial x} +2\frac{\partial g}{\partial y}-\frac{x^2}{3}-\frac{y^2}{3}\right)^2\,\rmd x \rmd y-8C\frac{\mathcal{B}_0^2}{n^2}.
\end{eqnarray*}
\end{enumerate}
Therefore, since  $\mathcal{S}[\mathbf{x}]$  can be rewritten in the following form
\begin{eqnarray*}
\mathcal{S}[\mathbf{x}]&=&\int_{B}\left(\frac{1}{1-\nu}\tr(\gamma)^2-2\det(\gamma)\right)\,\rmd x \rmd y\\
&=&\int_{B}\left(\frac{\nu}{1-\nu}\left(\gamma_{11}+\gamma_{22}\right)^2+\gamma_{11}^2+2\gamma_{12}^2+\gamma_{22}^2\right)\,\rmd x \rmd y\\
&\geq& \int_{B^{r^*}}\left(\frac{\nu}{1-\nu}\left(\gamma_{11}+\gamma_{22}\right)^2+\gamma_{11}^2+2\gamma_{12}^2+\gamma_{22}^2\right)\,\rmd x \rmd y
\end{eqnarray*} 
it follows by items i-iv that there exists a constant $C$ independent of $n$ and $\mathbf{x}$ such that
\begin{equation*}
\mathcal{S}[\mathbf{x}]\geq \frac{1}{2}\mathcal{F}-C\frac{\mathcal{B}_0^2}{n^2}.
\end{equation*}
\end{proof}

\begin{remark}
The preceding lemma is the essential estimate that quantifies the trade off between bending and stretching energy. Furthermore, it shows that in the limit $ n\rightarrow \infty $ that $ \mathcal{S} $ is bounded away from zero. That is, as the bending energy increases by adding more waves there is no reduction in the stretching energy. This is in contrast to the behaviour of a minimal ridge in which with decreasing thickness the bending energy diverges while the stretching energy converges to zero \cite{Shankar:Ridge, Conti2008}.
\end{remark}

\begin{theorem} \label{perconfig:EnergyScalingTwo}
Suppose $n\in \{2,3,\ldots\}$ and $\tau>0$. There exists constants $c,C>0$ independent of $n$ such that
\begin{equation*}
 \min\{cn\tau^2,\mathcal{F}/2\}\leq \inf_{\mathbf{x}\in \mathcal{A}_n}\mathcal{E}[\mathbf{x}]\leq \min\{Cn^2\tau^2,\mathcal{F}\}.
\end{equation*}
\end{theorem}
\begin{proof}
By lemma \ref{perconfig:lower-bound-stretch} there exists a constant $c_1$ independent of $n$ and $\mathbf{x}$ such that
\begin{equation*}
\mathcal{S}[\mathbf{x}]\geq\frac{1}{2}\mathcal{F}-c_1\frac{\mathcal{B}[\mathbf{x}]^2}{n^2}.
\end{equation*}
Therefore, minimizing the function $\mathcal{E}=\mathcal{S}+\tau^2\mathcal{B}$ subject to the constraints
\begin{equation*}
\begin{array}{ccc}
\mathcal{S}\geq 0, & \tau^2\mathcal{B}\geq 0, & \mathcal{S}\geq\frac{1}{2}\mathcal{F}-c_1\frac{\mathcal{B}}{n^2},
\end{array}
\end{equation*}
it follows that there exists a constant $c$ independent of $n$ and $\mathbf{x}$ such that
\begin{equation*}
\max\left\{cn\tau^2,\mathcal{F}/2\right\}\leq\inf_{\mathbf{x}\in \mathcal{A}_n}\mathcal{E}[\mathbf{x}].
\end{equation*}
The upper bound follows from lemma \ref{isoimmersion:periodicupperbound}.
\end{proof}

\begin{corollary} \label{scaling:cor1}
There exists $ n^*\geq 2 $ such that if $ \tau< \left[\mathcal{F}/(4\pi(1-r_0^2))\right]^{1/2} $ and $ n>n^* $ then
\begin{equation*}
\inf_{\mathbf{x}\in \mathcal{A}}\mathcal{E}_{\tau}[\mathbf{x}]< \inf_{\mathbf{x}\in \mathcal{A}_n}\mathcal{E}_{\tau}[\mathbf{x}].
\end{equation*}
\end{corollary}
\begin{proof}
By theorem \ref{perconfig:EnergyScalingTwo} it follows that there exists a constant $ c $ independent of $ n $ and $ \tau $ such that $
\min\{cn\tau^2,\mathcal{F}/2\}\leq \inf_{\mathbf{x}\in \mathcal{A}_n}\mathcal{E}_{\tau}[\mathbf{x}].$
Furthermore, if $ \tau<  \left[\mathcal{F}/(4\pi(1-r_0^2))\right]^{1/2}$ then $ 2\pi(1-r_0^2)\tau^2< \mathcal{F}/2 $ and if $ n>2\pi(1-r_0^2)/c$ then $ 2\pi(1-r_0)^2<cn$. Therefore, letting $ n^*=2\pi(1-r_0^2)/c$ it follows that if $ \tau<  \left[\mathcal{F}/(4\pi(1-r_0^2))\right]^{1/2}$ and $ n>n^* $ then by proposition \ref{isoimmersion:upperbound}
\begin{equation*}
\inf_{\mathbf{x}\in \mathcal{A}}\mathcal{E}_{\tau}[\mathbf{x}]\leq 2\pi(1-r_0^2)\tau^2\leq \min\{\mathcal{F}/2,cn\tau^2\}< \inf_{\mathbf{x}\in \mathcal{A}_n}\mathcal{E}_{\tau}[\mathbf{x}].
\end{equation*}

By theorem \ref{perconfig:EnergyScalingTwo} and the fact that the upper bound in proposition \ref{isoimmersion:upperbound} corresponds to the elastic energies of deformations in $ \mathcal{A}_2$ it follows that $
\min\{2c\tau^2,\mathcal{F}/2\}\leq 2\pi(1-r_0)^2\tau^2. $ Therefore, $ c\leq \pi(1-r_0^2) $ and consequently $ n^*=2\pi(1-r_0^2)/c\geq 2.$
\end{proof}

\begin{corollary} \label{scaling:cor2}
Let $ n \in \{2,3,\ldots\}$. There exists constants $ c,C>0$ independent of $ n $ and $ \tau $ such that if $ \tau< \mathcal{F}/(2c)n^{-1/2}$ then
\begin{equation*}
cn\tau^2 \leq \inf_{\mathbf{x}\in \mathcal{A}_n}\mathcal{E}_{\tau}[\mathbf{x}]\leq Cn^2 \tau^2.
\end{equation*}
\end{corollary}
\begin{proof}
By theorem \ref{perconfig:EnergyScalingTwo} it follows that there exists constants $ c,C>0 $ such that  
\begin{equation*}
 \min\{cn\tau^2,\mathcal{F}/2\}\leq \inf_{\mathbf{x}\in \mathcal{A}_n}\mathcal{E}[\mathbf{x}]\leq \min\{Cn^2\tau^2,\mathcal{F}\}\leq Cn^2\tau^2
\end{equation*}
Therefore, if $ \tau<\mathcal{F}/(2c) n^{-1/2}$ it follows that $ \min\{cn\tau^2,\mathcal{F}/2 \}=cn\tau^2$ and the result follows.
\end{proof}

The previous corollaries \ref{scaling:cor1} and \ref{scaling:cor2} extend the results of theorem \ref{perconfig:EnergyScalingTwo} and quantify different ``crossover regimes'' in $ n $ and $ \tau $. Specifically, corollary \ref{scaling:cor1} gives a critical wave number $ n^*\geq 2 $ such that energetically there can be no refinement with decreasing thickness of the number of waves greater than $ n^* $. Corollary \ref{scaling:cor2} gives a crossover condition for when minimizers transition from being close to a flat deformation to one whose elastic energy scales with $ \tau $ like an isometric immersion. In figure
\ref{fig:crossover} we plot a schematic of these crossover regimes.
\begin{figure}
\begin{center}
\includegraphics[width=5.0in]{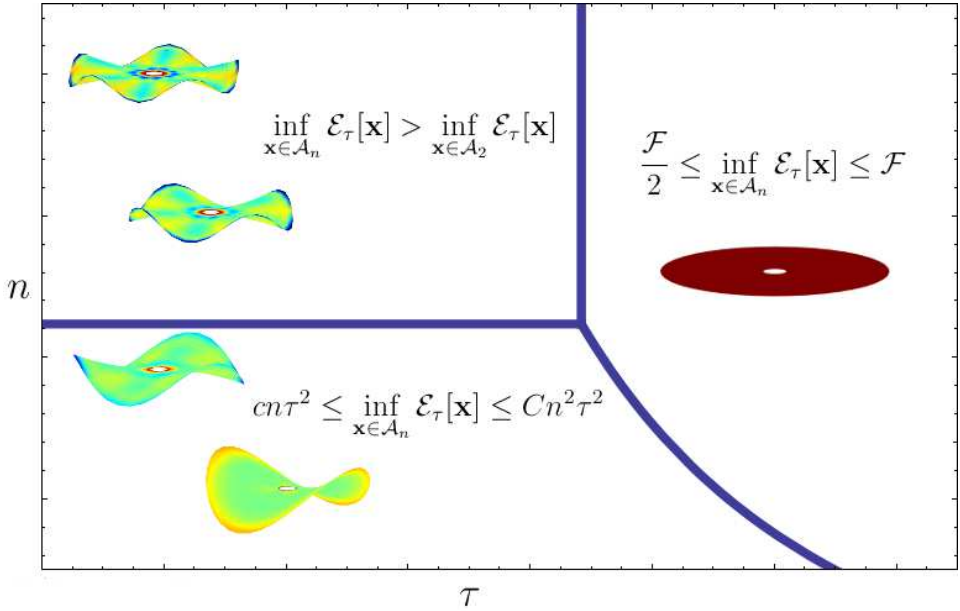}
\end{center}
\caption{A schematic of regions in the $ n-\tau $ plane in which minimizing deformations in $ \mathcal{A}_n $ ``crossover'' to different scaling regimes. In the upper left region minimizing deformations in $ \mathcal{A}_n $ have larger energy than minimizers over $ \mathcal{A} $. the existence of this upper bound means that there can be no refinement with decreasing thickness of the number of waves for minimizers in $ \mathcal{A} $. In the lower left region the minimizing deformations scale in energy like that of an isometric immersion. In the region on the right the minimizing deformations are either a flat deformation or are close to being a flat deformation. It is important to note that this figure is just a schematic drawn for a specific values of the constant  $ c $ in theorem \ref{perconfig:EnergyScalingTwo}. Different values of $ c $ will change the positions of the boundaries between the regions but will not change the qualitative form of the figure. }
\label{fig:crossover}
\end{figure}
%
%
%
\section{Boundary layer analysis}

The $n$-periodic isometric immersions constructed in section \ref{isoimmersions} satisfy the FvK equations (\ref{FvK:firstFVK}) and (\ref{FvK:secondFVK}) but not the boundary conditions (\ref{FvK:outofplaneBC1}) and (\ref{FvK:outofplaneBC2}). But, in figure \ref{fig:periodicconvergence} we see that in the vanishing thickness limit numerical minimizers of $\mathcal{E}_{\tau}$ over $\mathcal{A}_n$ converge to an isometric immersion. Therefore, we expect the sheet to be a perturbation of an isometric immersion that introduces boundary layers near the edge of the disk and along lines of inflection. By theorem \ref{perconfig:EnergyScalingTwo} the width of these regions will scale with $\tau$ so that the energy in these regions scales like $\tau^2$ or smaller.


Again, let $n\in\{2,3,\ldots\}$. Define $S_n$ to be the sector in $\mathbb{R}^2$ bounded between $\theta=0$ and $\theta=\pi/n$ and define $B_n=B\cap S_n$. Let $\mathbf{x}\in \mathcal{A}_n$ with corresponding out-of-plane displacement $\eta$ and potential $\Phi$ and suppose $\eta$ and $\Phi$ are perturbations of an $n$-periodic isometric immersion. That is, on $B_n$ assume $\eta$ and $\Phi$ are of the form
\begin{eqnarray}
\eta =y\left(x-\cot\left(\frac{\pi}{n}\right)y\right)+\tilde{\eta}, \label{BoundaryLayer:ansatz-eta}\\
\Phi =\tilde{\Phi}, \label{BoundaryLayer:ansatz-Phi}
\end{eqnarray}
for some perturbations $\tilde{\Phi}$ and $\tilde{\eta}$. Furthermore, since $\mathbf{x}\in \mathcal{A}_n$ we enforce the Dirichlet boundary conditions $\tilde{\eta}=0$ along the lines $\theta=0$ and $\theta=\frac{\pi}{n}$.


The elastic energy of this perturbation is
\begin{eqnarray}\label{BoundaryLayer:Energy}
\fl \mathcal{E}_{\tau}[\mathbf{x}]=2n\int_{B_n}\frac{1}{1+\nu}\left(\Delta \tilde{\Phi}\right)^2\,dxdy +2n\tau^2\int_{B_n}\left[\frac{1}{1-\nu}\left(\Delta \tilde{\eta}-2\cot\left(\frac{\pi}{n}\right)\right)^2\right.\nonumber\\
\left.+4\cot\left(\frac{\pi}{n}\right)\frac{\partial^2 \tilde{\eta}}{\partial x^2}+4\frac{\partial^2\tilde{\eta}}{\partial x \partial y}-2[\tilde{\eta},\tilde{\eta}]+2\right]\,\rmd x \rmd y.
\end{eqnarray}
If $\tilde{\eta}$ and $\tilde{\Phi}$ extremize the $\mathcal{E}_{\tau}$ then by the FvK equations (\ref{FvK:firstFVK}-\ref{FvK:outofplaneBC2}), and the fact that $\delta \eta=0$ along the lines $\theta=0$ and $\theta=\pi/n$, it follows that $\tilde{\eta}$ and $\tilde{\Phi}$ satisfy the following boundary value problem:
\begin{eqnarray} \label{BoundaryLayer:Eq1}
\frac{1}{2(1+\nu)}\Delta^2 \tilde{\Phi}-2\cot\left(\frac{\pi}{n}\right)\frac{\partial^2 \tilde{\eta}}{\partial x^2}-2\frac{\partial^2 \tilde{\eta}}{\partial x \partial y}+[\tilde{\eta},\tilde{\eta}]=0,\\
\frac{\tau^2}{4(1-\nu)}\Delta^2 \tilde{\eta}+\cot\left(\frac{\pi}{n}\right)\frac{\partial^2 \tilde{\Phi}}{\partial x^2}+\frac{\partial^2 \tilde{\Phi}}{\partial x \partial y}+[\tilde{\Phi},\tilde{\eta}]=0, \label{BoundaryLayer:Eq2}
\end{eqnarray}
\begin{eqnarray} \label{BoundaryLayer:StressBC}
\begin{array}{cc}
\displaystyle{\left.\frac{1}{r}\frac{\partial^2 \tilde{\Phi}}{\partial \theta^2}+\frac{\partial \tilde{\Phi}}{\partial r}\right|_{r=r_0,1}=0}, & 
 \displaystyle{\left.\frac{\partial^2 \tilde{\Phi}}{\partial r \partial \theta}-\frac{1}{r}\frac{\partial \tilde{\Phi}}{\partial \theta}\right|_{r=r_0,1}=0},
 \end{array}
 \end{eqnarray}
 \begin{equation} \label{BoundaryLayer:EdgeBC1}
\fl \left.\frac{1}{1-\nu}\Delta \tilde{\eta} -\mathbf{n}^T\cdot D^2 \tilde{\eta} \cdot \mathbf{n}\right|_{r=r_0,1}= -\csc\left(\frac{\pi}{n}\right)\left(\frac{(\nu-2)\cos\left(\frac{\pi}{n}\right)+\nu\cos\left(\frac{\pi}{n}-2\theta\right)}{1-\nu}\right), 
\end{equation}
\begin{equation}
\fl \left.\frac{1}{1-\nu}\frac{\partial \Delta \tilde{\eta}}{\partial \mathbf{n}}+\frac{\partial}{\partial \mathbf{t}}\left(\mathbf{n}^T\cdot D^2 \tilde{\eta} \cdot \mathbf{t}\right)\right|_{r=r_0,1}=\left.\frac{2}{r}\csc\left(\frac{\pi}{n}\right)\cos\left(\frac{\pi}{n}-2\theta\right)\right|_{r=r_0,1}, \label{BoundaryLayer:EdgeBC2}
\end{equation} 
\begin{equation}
\begin{array}{cc}
\displaystyle{\left.\mathbf{n}\cdot\left(\frac{\partial^2 \tilde{\Phi}}{\partial y^2},-\frac{\partial^2 \tilde{\Phi}}{\partial x \partial y}\right)\right|_{\theta=0,\frac{\pi}{n}}=0,} & \displaystyle{\left. \mathbf{n}\cdot \left(-\frac{\partial^2 \tilde{\Phi}}{\partial x \partial y},\frac{\partial^2 \tilde{\Phi}}{\partial x^2}\right)\right|_{\theta=0,\frac{\pi}{n}}=0,} \label{BoundaryLayer:LineStressBC}
\end{array}
\end{equation}
\begin{equation}
\begin{array}{cc}
\displaystyle{\left.\tilde{\eta}\right|_{\theta=0,\pi/n}=0,} & 
 \displaystyle{\left.\frac{\partial^2 \tilde{\eta}}{\partial \theta^2}\right|_{\theta=0,\pi/n}=2r^2\cot\left(\frac{\pi}{n}\right)}.\label{BoundaryLayer:LineBC} 
 \end{array}
 \end{equation}

\begin{figure}[ht]
\begin{center}
\includegraphics[width=2.5in]{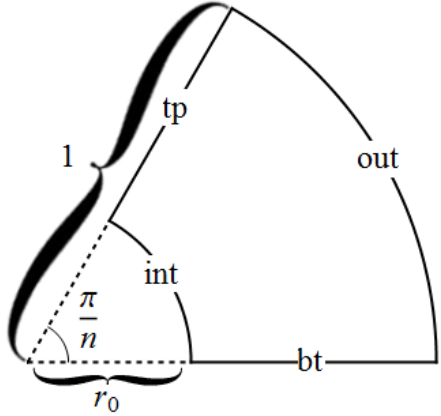}
\end{center}
\caption{The domain for the governing equations of the perturbation from an isometric immersion is $B_n=\{(r,\theta):r_0<r<1, 0<\theta<\pi/n \}$.}
\label{fig:domain}
\end{figure}

\begin{remark} So far no approximations have taken place. The boundary value problem (\ref{BoundaryLayer:Eq1}-\ref{BoundaryLayer:LineBC}) is the Euler-Lagrange equations corresponding to the ansatz (\ref{BoundaryLayer:ansatz-eta}) and (\ref{BoundaryLayer:ansatz-Phi}).
\end{remark}

\begin{remark} \label{BoundaryLayer:n2remark}
The  boundary value problem (\ref{BoundaryLayer:Eq1}-\ref{BoundaryLayer:LineBC}) for $ n=2 $ is different from the other cases since $ \cot(\pi/2)=0$. We will see that for the case $ n=2 $ that there is no need to introduce a boundary layer near the lines of inflection. Furthermore, the boundary layers near the edges of the annulus will have a different geometry as well. 
\end{remark}

We look for approximate solutions of (\ref{BoundaryLayer:Eq1}) and (\ref{BoundaryLayer:Eq2}) that are linear combinations of boundary layer solutions near the interior radius (int), outer radius (out), bottom of the sector (bt), and the top of the sector (tp) (see figure \ref{fig:domain}). That is we assume 
\begin{equation}
\tilde{\eta}=\tilde{\eta}_{int}+\tilde{\eta}_{out}+\tilde{\eta}_{bt}+\tilde{\eta}_{\text{tp}} \text{ and } \tilde{\Phi}=\tilde{\Phi}_{int}+\tilde{\Phi}_{out}+\tilde{\Phi}_{bt}+\tilde{\Phi}_{\text{tp}},
\end{equation}
where each term is found by an appropriate asymptotic expansion. The full configuration with domain $B$ can then be obtained by taking odd extensions of $\eta$ and even extensions of $\Phi$. 

\begin{remark}
To construct a complete asymptotic solution of the boundary value problem (\ref{BoundaryLayer:Eq2}-\ref{BoundaryLayer:LineBC}) we would also need to analyse regions where boundary layers overlap. In this paper we are interested only in the scaling of the width of the boundary layer and do not analyse these overlap regions in this work.
\end{remark}


\subsection{Boundary layer near outer radius}
Define the rescaled radius $\tilde{r}_{out}$ and functions $\tilde{\eta}_{out}^{\prime}$ and $\tilde{\Phi}_{out}^{\prime}$ by
\begin{equation*}
\begin{array}{ccc}
\tilde{r}_{out}=\tau^{\alpha}(1-r), & \tilde{\eta}_{out} =\tau^{\beta}\tilde{\eta}_{out}^{\prime}, & \tilde{\Phi}_{out}=\tau^{\gamma} \tilde{\Phi}^{\prime}_{out}
\end{array}
\end{equation*}
where $-\alpha,\beta,\gamma\in \mathbb{R}^+$. To lowest order the stretching and bending energies near the outer radius are
\begin{eqnarray*}
\fl \frac{\mathcal{S}[\mathbf{x}]}{2n}=\tau^{2\gamma+4\alpha}\int_{B_n}\frac{1}{1+\nu}\left(\frac{\partial^2 \tilde{\Phi}_{out}^{\prime}}{\partial \tilde{r}_{out}^2}\right)^2\,\rmd x \rmd y,\\
\fl \frac{\tau^2\mathcal{B}[\mathbf{x}]}{2n}= \tau^2\int_{B_n}\left[\frac{1}{1-\nu}\left(\tau^{\beta+2\alpha}\frac{\partial^2 \tilde{\eta}_{out}^{\prime}}{\partial \tilde{r}_{out}^2}-2\cot\left(\frac{\pi}{n}\right)\right)^2\right. \nonumber \\
\left.+4\tau^{\beta+2\alpha}\cos(\theta)\left(\cot\left(\frac{\pi}{n}\right)+\sin(\theta)\right)\frac{\partial^2 \tilde{\eta}_{out}^{\prime}}{\partial \tilde{r}_{out}^2}+2\right]\,\rmd x \rmd y,
\end{eqnarray*}
and the compatibility condition is
\begin{equation*}
\frac{1}{2(1+\nu)}\tau^{\gamma+4\alpha}\frac{\partial^4 \tilde{\Phi}_{out}^{\prime}}{\partial \tilde{r}_{out}^4}-2\tau^{\beta+2\alpha}\cos(\theta)\left(\cot(\pi/n)+\sin(\theta)\right)\frac{\partial^2 \tilde{\eta}_{out}^{\prime}}{\partial\tilde r_{out}^2}=0.
\end{equation*}

To ensure that the elastic energy is $\Or(\tau^2)$ and the compatibility condition is non-trivial we must have that
\begin{equation*}
\begin{array}{ccc}
\beta+2\alpha=0, & 2\gamma+4\alpha=2, & \gamma+4\alpha=\beta+2\alpha.
\end{array}
\end{equation*}
The solution to these equations gives us the following scaling
\begin{equation*}
\begin{array}{ccc}
\alpha=-\frac{1}{2}, & \beta=1, & \gamma=2,
\end{array}
\end{equation*}
which motivates the asymptotic expansion
\begin{eqnarray*}
\tilde{\eta}_{out}=\tau\eta_{out}^{(0)}+\tau^{\frac{3}{2}}\eta_{out}^{(1)}+\tau^{2}\eta_{out}^{(2)}+\ldots,\\
\tilde{\Phi}_{out}=\tau^2\Phi_{out}^{(0)}+\tau^{\frac{5}{2}}\Phi_{out}^{(1)}+\tau^3\Phi_{out}^{(2)}+\ldots.
\end{eqnarray*}
Therefore, to lowest order the equations for the perturbation (\ref{BoundaryLayer:Eq1})  and (\ref{BoundaryLayer:Eq2}) become
\begin{equation}\label{BoundaryLayer:outeqn1}
\frac{\partial^4\Phi_{out}^{(0)}}{\partial \tilde{r}_{out}^4}-4(1+\nu)\csc\left(\frac{\pi}{n}\right)\cos(\theta)\cos\left(\frac{\pi}{n}-\theta\right)\frac{\partial^2 \eta_{out}^{(0)}}{\partial \tilde{r}_{out}^2}=0,
\end{equation}
\begin{equation}\label{BoundaryLayer:outeqn2}
\frac{\partial^4\eta_{out}^{(0)}}{\partial\tilde{r}_{out}^4}+4(1-\nu)\csc\left(\frac{\pi}{n}\right)\cos(\theta)\cos\left(\frac{\pi}{n}-\theta\right)\frac{\partial^2 \Phi_{out}^{(0)}}{\partial \tilde{r}_{out}^2}=0.
\end{equation}
Furthermore, the boundary condtions (\ref{BoundaryLayer:StressBC}), (\ref{BoundaryLayer:EdgeBC1}) and (\ref{BoundaryLayer:EdgeBC2}) to lowest order become
\begin{eqnarray}\label{BoundaryLayer:outeqnBC1}
\left.\frac{\partial\Phi_{out}^{(0)}}{\partial \tilde{r}_{out}}\right|_{\tilde{r}=0}=0, \qquad \left.\frac{\partial^3\eta_{out}^{(0)}}{\partial\tilde{r}_{out}^3}\right|_{\tilde{r}=0}=0,\\
\left.\frac{\partial^2 \eta_{out}^{(0)}}{\partial \tilde{r}_{out}^2}\right|_{\tilde{r}_{out}=0}=-2\nu\csc\left(\frac{\pi}{n}\right)\cos(\theta)\cos\left(\frac{\pi}{n}-\theta\right)+2\cot\left(\frac{\pi}{n}\right).\label{BoundaryLayer:outeqnBC2}
\end{eqnarray}

To solve the boundary value problem (\ref{BoundaryLayer:outeqn1}-\ref{BoundaryLayer:outeqnBC2}) we make the following ansatz
\begin{eqnarray*} \eta_{out}^{(0)}=\lambda_n^3(\theta)A(\tilde{r}_{out}\lambda_n^{-1}(\theta))-B(\tilde{r}_{out},\theta),
\end{eqnarray*}
\begin{eqnarray*}
\Phi_{out}^{(0)}=\lambda_n^6(\theta)C(\tilde{r}_{out}\lambda_n^{-1}(\theta))-D(\tilde{r}_{out},\theta),
\end{eqnarray*}
where $\lambda_n:\left(0,\frac{\pi}{n}\right)\rightarrow \mathbb{R}$ is defined by
\begin{equation*}
\lambda_n(\theta)=\csc\left(\frac{\pi}{n}\right)\cos(\theta)\cos\left(\frac{\pi}{n}-\theta\right).
\end{equation*}
The solution to the resulting differential equations containing exponentially decaying terms and is given by
\begin{equation*}
\fl
A(\tilde{r}_{out}\lambda_n^{-1}(\theta))=\frac{\nu}{\sqrt{2(1-\nu^2)}}\exp\left(-\frac{\sqrt{2}\tilde{r}_{out}(1-\nu^2)^{\frac{1}{2}}}{\lambda_n(\theta)}\right)\sin\left(\frac{\sqrt{2}\tilde{r}_{out}(1-\nu^2)^{\frac{1}{4}}}{\lambda_n(\theta)}-\frac{\pi}{4}\right),
\end{equation*}
\begin{eqnarray*}
\fl B(\tilde{r}_{out},\theta)&=&\frac{\cot(\pi/n)}{\lambda_n(\theta)\sqrt{2(1-\nu^2)}}\exp\left(-\sqrt{2}\tilde{r}_{out}(1-\nu^2)^{\frac{1}{4}}\lambda_n^{\frac{1}{2}}(\theta)\right)\\
\fl &\,& \times \sin\left(\sqrt{2}\tilde{r}_{out}(1-\nu^2)^{\frac{1}{4}}\lambda_n^{\frac{1}{2}}(\theta)-\frac{\pi}{4}\right),
\end{eqnarray*}
\begin{equation*}
\fl C(\tilde{r}_{out}\lambda_n^{-1}(\theta))=\frac{\nu}{2(1-\nu)}\exp\left(-\frac{\sqrt{2}\tilde{r}_{out}(1-\nu^2)^{\frac{1}{4}}}{\lambda_n(\theta)}\right)\cos\left(\frac{\sqrt{2}\tilde{r}(1-\nu^2)^{\frac{1}{4}}}{\lambda_n(\theta)}-\frac{\pi}{4}\right),
\end{equation*}
\begin{eqnarray*}
\fl D(\tilde{r}_{out},\theta)&=&\frac{\cot(\pi/n)}{\lambda_n(\theta)\sqrt{2}(1-\nu)}\exp\left(-\sqrt{2}\tilde{r}(1-\nu^2)^{\frac{1}{4}}\lambda_n^{\frac{1}{2}}(\theta)\right)\\
\fl &\,& \times \cos\left(\sqrt{2}\tilde{r}_{out}(1-\nu^2)^{\frac{1}{4}}\lambda_n^{\frac{1}{2}}(\theta)-\frac{\pi}{4}\right).
\end{eqnarray*}

$\lambda_n^3(\theta)A(\tilde{r}_{out}\lambda_n^{-1}(\theta))$ and $\lambda_n^6(\theta)C(\tilde{r}_{out}\lambda_n^{-1}(\theta))$ can be interpreted as terms that alone would reduce the magnitude of the Gaussian curvature in a thin boundary layer while $B(\tilde{r}_{out},\theta)$ and $ D(\tilde{r}_{out},\theta) $ would alone locally reduce the mean curvature in a separate boundary layer with a different geometry. In the overlap of these two boundary layers the bending energy is reduced through the combination of these two effects.  By expressing $ \eta_{out}^{(0)} $ and $ \Phi_{out}^{(0)} $ in terms of the actual radius $ \rho=Rr=R\tau^{1/2}\tilde{r} $ we have the following results for the scaling of the width of these boundary layers:
\begin{enumerate}
\item The width of the \emph{boundary layer in which the Gaussian curvature is significantly reduced} satisfies the following scaling
\begin{equation*}
\text{width}(\theta)_{\rho=R}\sim t^{\frac{1}{2}}|K_0|^{-\frac{1}{4}}\csc(\pi/n)\cos(\theta)\cos(\pi/n-\theta).
\end{equation*}
\item For $ n\geq 3$, the width of the \emph{boundary layer in which the mean curvature is significantly reduced} satisfies the following scaling
\begin{equation*}
\text{width}(\theta)_{\rho=R}\sim t^{\frac{1}{2}}|K_0|^{-\frac{1}{4}}\sqrt{\sin(\pi/n)\sec(\theta)\sec(\pi/n-\theta)}.
\end{equation*}
\end{enumerate}

\subsection{Boundary layer near interior radius}
The boundary layer near $r=r_0$ is completely analogous to the one near $r=1$.
Define the rescaled radius $\tilde{r}_{int}$ by
\begin{equation*}
\tilde{r}_{int}=\tau^{-\frac{1}{2}}(r-r_0) 
\end{equation*}
and consider the asymptotic expansion
\begin{eqnarray*}
\tilde{\eta}_{int}=\tau\eta_{int}^{(0)}+\tau^{\frac{3}{2}}\eta_{int}^{(1)}+\tau^{2}\eta_{int}^{(2)}+\ldots\\
\tilde{\Phi}_{int}=\tau^2\Phi_{int}^{(0)}+\tau^{\frac{5}{2}}\Phi_{int}^{(1)}+\tau^3\Phi_{int}^{(2)}+\ldots.
\end{eqnarray*}
This yields identical governing and boundary equations as the outer radius and thus
\begin{eqnarray*}
 \eta_{int}^{(0)}=\lambda_n^3(\theta)A(\tilde{r}_{int},\lambda_n^{-1}(\theta))-B(\tilde{r}_{int},\theta)
\end{eqnarray*}
\begin{eqnarray*}
 \Phi_{int}^{(0)}=\lambda_n^6(\theta)A(\tilde{r}_{int}),\lambda_n^{-1}(\theta)-D(\tilde{r}_{int},\theta),
\end{eqnarray*}
where $ A,B,C,D $ are defined as in the previous subsection.
%
%
\subsection{Boundary layer near the bottom of the sector}
From the observations in the remark \ref{BoundaryLayer:n2remark}  we will assume in this section that $n\in\{3,4,\ldots\}.$
%
%
Define the rescaled coordinate $\tilde{y}_{bt}$ and functions $\tilde{\eta}^{\prime}_{bt}$ and $\tilde{\Phi}^{\prime}_{bt}$ by

\begin{equation*}
\begin{array}{ccc}
\tilde{y}_{bt}=\tau^{\alpha}y, & \tilde{\eta}_{bt} =\tau^{\beta}\tilde{\eta}_{bt}^{\prime}, & \tilde{\Phi}_{bt}=\tau^{\gamma} \tilde{\Phi}^{\prime}_{bt}
\end{array}
\end{equation*}
where $-\alpha,\beta,\gamma\in \mathbb{R}^+$. To lowest order the stretching and bending energies near $y=0$ are
\begin{eqnarray*}
\frac{\mathcal{S}[\mathbf{x}]}{2n}=\tau^{2\gamma+4\alpha}\int_{B_n}\frac{1}{1+\nu}\left(\frac{\partial^2 \tilde{\Phi}_{bt}^{\prime}}{\partial \tilde{y}_{bt}^2}\right)^2\,\rmd x \rmd y,\\
\frac{\tau^2\mathcal{B}[\mathbf{x}]}{2n}= \tau^2\int_{B_n}\left[\frac{1}{1-\nu}\left(\tau^{\beta+2\alpha}\frac{\partial^2 \tilde{\eta}_{bt}^{\prime}}{\partial \tilde{y}_{bt}^2}-2\cot\left(\frac{\pi}{n}\right)\right)^2+2\right]\,\rmd x \rmd y,
\end{eqnarray*}
and the compatibility condition is
\begin{equation*}
\frac{1}{2(1+\nu)}\tau^{\gamma+4\alpha}\frac{\partial^4 \tilde{\Phi}_{bt}^{\prime}}{\partial \tilde{y}_{bt}^4}-2\tau^{\beta+\alpha}\frac{\partial^2 \tilde{\eta}_{bt}}{\partial x \partial \tilde{y}_{bt}}=0.
\end{equation*}

Therefore, the scaling that ensures the elastic energy is $\Or(\tau^2)$ and the compatibility equation is non-trivial is
\begin{equation*}
\begin{array}{ccc}
\alpha=-\frac{1}{3}, & \beta=\frac{2}{3}, & \gamma=\frac{5}{3},
\end{array}
\end{equation*}
which is a different scaling then the one near the edges of the annulus. This scaling motivates the asymptotic expansion
\begin{eqnarray*}
\tilde{\eta}_{bt}=\tau^{\frac{2}{3}}\eta_{bt}^{(0)}+\tau\eta_{bt}^{(1)}+\tau^{\frac{4}{3}}\eta_{bt}^{(2)}+\ldots\\
\tilde{\Phi}_{bt}=\tau^{\frac{5}{3}}\Phi_{bt}^{(0)}+\tau^{2}\Phi_{bt}^{(1)}+\tau^{\frac{7}{3}}\Phi_{bt}^{(2)}+\ldots.
\end{eqnarray*}
Consequently, to lowest order (\ref{BoundaryLayer:Eq1}) and (\ref{BoundaryLayer:Eq2}) become
\begin{equation}\label{BoundaryLayer:bteqn1}
\frac{\partial}{\partial \tilde{y}_{bt}}\left(\frac{\partial^3 \Phi_{bt}^{(0)}}{\partial \tilde{y}_{bt}^3}-4(1+\nu)\frac{\partial \eta_{bt}^{(0)}}{\partial x}\right) =0, \qquad
\end{equation}
\begin{equation}\label{BoundaryLayer:bteqn2}
\frac{\partial}{\partial \tilde{y}_{bt}}\left(\frac{\partial^3 \eta_{bt}^{(0)}}{\partial \tilde{y}_{bt}^3}+4(1-\nu)\frac{\partial \Phi_{bt}^{(0)}}{\partial x}\right)=0.
\end{equation}
Furthermore, by (\ref{BoundaryLayer:LineStressBC}) and (\ref{BoundaryLayer:LineBC}) the boundary conditions are
\begin{eqnarray}
\left.\frac{\partial^2\Phi_{bt}^{(0)}}{\partial x^2}\right|_{\tilde{y}_{bt}=0}=0, \qquad \left.\frac{\partial^2 \Phi_{bt}^{(0)}}{\partial x \partial \tilde{y}_{bt}}\right|_{\tilde{y}_{bt}=0}=0,\label{BoundaryLayer:bteqnBC1}\\
\left.\eta_{bt}^{(0)}\right|_{\tilde{y}_{bt}=0}=0, \qquad
\left.\frac{\partial^2 \eta_{bt}^{(0)}}{\partial \tilde{y}_{bt}^2}\right|_{\tilde{y}_{bt}=0}=2\cot\left(\frac{\pi}{n}\right). \label{BoundaryLayer:bteqnBC2}
\end{eqnarray}

To solve this boundary value problem  we make the following ansatz
\begin{equation}
\begin{array}{c}
\tilde{\eta}_{bt}^{(0)}=\psi_1(x)+\tilde{y}_{bt}^2\cot\left(\frac{\pi}{n}\right)f(x,\tilde{y}_{bt}),\\
\tilde{\Phi}_{bt}^{(0)}=\psi_2(x)+\tilde{y}_{bt}^2\cot\left(\frac{\pi}{n}\right)g(x,\tilde{y}_{bt}),
\end{array}
\end{equation}
which transforms equations (\ref{BoundaryLayer:bteqn1}) and (\ref{BoundaryLayer:bteqn2}) into
\begin{equation}
\begin{array}{c}
\frac{\partial}{\partial \tilde{y}_{bt}}\left(\tilde{y}_{bt}^2\frac{\partial^3 g}{\partial \tilde{y}_{bt}^3}+6\tilde{y}_{bt}\frac{\partial^2 g}{\partial \tilde{y}_{bt}^2}+6\frac{\partial g}{\partial \tilde{y}_{bt}}-4(1+\nu)\tilde{y}_{bt}^2\frac{\partial f}{\partial x}\right)=0,\\
\frac{\partial}{\partial \tilde{y}_{bt}}\left(\tilde{y}_{bt}^2\frac{\partial^3 f}{\partial \tilde{y}_{bt}^3}+6\tilde{y}_{bt}\frac{\partial^2 f}{\partial \tilde{y}_{bt}^2}+6\frac{\partial f}{\partial \tilde{y}_{bt}}+4(1-\nu)\tilde{y}_{bt}^2\frac{\partial g}{\partial x}\right)=0.
\end{array}
\end{equation}
If we integrate with respect to $\tilde{y}_{bt}$  and make the similarity transformation $z=\tilde{y}_{bt}^3/x$, we have the following ordinary differential equations
\begin{equation*}
\begin{array}{c}
27z^2\frac{\rmd ^3g}{\rmd z^3}+108z\frac{\rmd ^2g}{\rmd z^2}+60\frac{\rmd g}{\rmd z}+4(1+\nu)z\frac{\rmd f}{\rmd z}=0,\\
27z^2\frac{\rmd ^3f}{\rmd z^3}+108z\frac{\rmd ^2f}{\rmd z^2}+60\frac{\rmd f}{\rmd z}-4(1-\nu)z\frac{\rmd g}{\rmd z}=0.
\end{array}
\end{equation*}
Solving for $\frac{\rmd g}{\rmd z}$ we obtain the following single differential equation
\begin{equation} \label{BoundaryLayer:fdifeqn}
\fl 4(1-\nu^2)z\frac{\rmd f}{\rmd z}+\left(\frac{90}{z}\frac{\rmd f}{\rmd z}+2430 \frac{\rmd ^2f}{\rmd z^2}+4455z\frac{\rmd ^3f}{\rmd z^3}+\frac{3645}{2}z^2\frac{\rmd^4f}{\rmd z^4}+\frac{729}{4}z^3 \frac{\rmd^5f}{\rmd z^5}\right)=0.
\end{equation}

The general solution to equation (\ref{BoundaryLayer:fdifeqn}) that does not contain exponentially growing terms is
\begin{eqnarray*}
\fl f(z)&=&\int\left[c_1 z^{-\frac{3}{2}}\text{Ker}_{\frac{1}{3}}\left(4\cdot 3^{-\frac{3}{2}}\cdot(1-\nu^2)^{\frac{1}{4}}\sqrt{z}\right)+c_2z^{-\frac{3}{2}}\text{Kei}_{\frac{1}{3}}\left(4\cdot 3^{-\frac{3}{2}}\cdot(1-\nu^2)^{\frac{1}{4}}\sqrt{z}\right)\right]\,\rmd z\\
\fl &\,&+c_3,
\end{eqnarray*}
where $\text{Ker}_{1/3}$ and $\text{Kei}_{1/3}$ denote Kelvin functions of the second kind and $c_1$, $c_2$ and $c_3$ are arbitrary constants. Therefore, we have that
\begin{eqnarray*}
\fl \eta_{bt}^{(0)}=\,\psi_1(x)+\tilde{y}_{bt}^2\cot(\pi/n)\left(\int\left[c_1 z^{-\frac{3}{2}}\text{Ker}_{\frac{1}{3}}\left(4\cdot 3^{-\frac{3}{2}}\cdot(1-\nu^2)^{\frac{1}{4}}\sqrt{z}\right)\right.\right. \nonumber\\
\left.\left.+c_2z^{-\frac{3}{2}}\text{Kei}_{\frac{1}{3}}\left(4\cdot 3^{-\frac{3}{2}}\cdot(1-\nu^2)^{\frac{1}{4}}\sqrt{z}\right)\right]\,\rmd z+c_3\right),\\
\fl \Phi_{bt}^{(0)}=\psi_2(x)+\tilde{y}_{bt}^2\cot(\pi/n)\left(\sqrt{\frac{1+\nu}{1-\nu}}\int\left[c_2 z^{-\frac{3}{2}}\text{Ker}_{\frac{1}{3}}\left(4\cdot 3^{-\frac{3}{2}}\cdot(1-\nu^2)^{\frac{1}{4}}\sqrt{z}\right)\right.\right. \nonumber\\
\left.\left.-c_1z^{-\frac{3}{2}}\text{Kei}_{\frac{1}{3}}\left(4\cdot 3^{-\frac{3}{2}}\cdot(1-\nu^2)^{\frac{1}{4}}\sqrt{z}\right)\right]\,\rmd z+c_4\right),
\end{eqnarray*}
where $c_4$ is an arbitrary constant. 

Now, near $y=0$, we have that
\begin{eqnarray*}
\fl \eta_{bt}^{(0)}= \psi_1(x)+\tilde{y}_{bt}^2\cot\left(\frac{\pi}{n}\right)\left[2^{-\frac{17}{6}}3^{\frac{3}{2}}\Gamma\left(\frac{1}{3}\right)(1-\nu^2)^{-\frac{1}{12}}\left(c_2-c_1\right)\frac{x^{\frac{2}{3}}}{\tilde{y}_{bt}^2}\right.\nonumber\\
\fl \qquad \left.-2^{-\frac{13}{6}}3^{\frac{1}{2}}\Gamma\left(-\frac{1}{3}\right)(1-\nu^2)^{\frac{1}{12}}\left(c_1\left(1+\sqrt{3}\right)+c_2\left(1-\sqrt{3}\right)\right)\frac{x^{\frac{1}{3}}}{\tilde{y}_{bt}}+c_3+\ldots \right],\\
\fl \Phi_{bt}^{(0)}=\,\Psi_2(x)+\tilde{y}_{bt}^2\cot\left(\frac{\pi}{n}\right)\left[-2^{-\frac{17}{6}}3^{\frac{3}{2}}\Gamma\left(\frac{1}{3}\right)(1-\nu^2)^{-\frac{7}{12}}(1-\nu)\left(c_1+c_2\right)\frac{x^{\frac{2}{3}}}{\tilde{y}_{bt}^2}\right. \nonumber \\
 \fl \qquad \left.  -2^{-\frac{13}{6}}3^{\frac{1}{3}}\Gamma\left(-\frac{1}{3}\right)\left(1-\nu^2\right)^{\frac{7}{12}}\left(1-\nu\right)^{-1}\left(\left(-1+\sqrt{3}\right)c_1+\left(1+\sqrt{3}\right)c_2\right)\frac{x^{\frac{1}{3}}}{\tilde{y}_{bt}}\right.\\
 \fl \qquad \left. +c_4\ldots \right].
\end{eqnarray*}
Furthermore,
\begin{equation*}
\begin{array}{c}
\displaystyle{\lim_{z\rightarrow \infty}\int z^{-\frac{3}{2}}\text{Ker}_{\frac{1}{3}}\left(4\cdot 3^{-\frac{3}{2}}(1-\nu^2)^{\frac{1}{4}}\sqrt{z}\right)\,\rmd z=\frac{\pi}{2}\sqrt{2+\sqrt{3}}},\\
\displaystyle{\lim_{z\rightarrow \infty}\int z^{-\frac{3}{2}}\text{Kei}_{\frac{1}{3}}\left(4\cdot 3^{-\frac{3}{2}}(1-\nu^2)^{\frac{1}{4}}\sqrt{z}\right)\,\rmd z=\frac{\pi}{2}\sqrt{2-\sqrt{3}}}.
\end{array}
\end{equation*}
Therefore, to satisfy the boundary conditions (\ref{BoundaryLayer:bteqnBC1}) and (\ref{BoundaryLayer:bteqnBC2}) we must have that
\begin{equation}
\begin{array}{ll}
\psi_1(x)=-\frac{3^{\frac{3}{2}}\Gamma\left(\frac{1}{3}\right)\cot\left(\frac{\pi}{n}\right)}{2^{7/3}(1-\nu^2)^{\frac{1}{3}}\pi}x^{\frac{2}{3}}, & \psi_2(x)=-\frac{3(1-\nu^2)^{\frac{1}{6}}\Gamma\left(\frac{1}{3}\right)\cot\left(\frac{\pi}{n}\right)}{2^{\frac{7}{3}}\left(1-\nu\right)\pi}x^{\frac{2}{3}},\\
c_1=-\frac{1+\sqrt{3}}{\sqrt{6}(1-\nu^2)^{\frac{1}{4}}\pi}, & c_2=\frac{-1+\sqrt{3}}{3^{\frac{3}{4}}(1-\nu^2)^{\frac{1}{4}}\pi},\\
c_3=1, & c_4=-\sqrt{\frac{1+\nu}{3(1-\nu)}}.
\end{array}
\end{equation}

By expressing $ \eta_{bt}^{(0)} $ and $ \Phi_{bt}^{(0)} $ in terms of the actual Cartesian coordinates $ u=Rx $ and $ v=Ry $ and approximating $ \rho $ and $ \theta $ near $ y=0 $ by $ \rho\approx u $ and $ \theta\approx \ v/u $, it follows that the width of the boundary layer in this region scales like 
\begin{equation*}
\text{width}(\rho)_{\eta=0}\sim t^{\frac{1}{3}}\rho^{\frac{1}{3}}|K_0|^{-\frac{1}{6}}.
\end{equation*}
This boundary layer can be interpreted as a region in which the mean curvature is locally reduced while the change in energy contributed from the Gaussian curvature of the perturbation is of the order $\Or(\tau^{5/3}) $.

\subsection{Boundary layer near the top of the sector}
To construct the boundary layer near $ \theta=\pi/n $ we can simply rotate the asymptotic expansion for $ \tilde{\eta}_{bt} $ and $ \tilde{\Phi}_{bt} $ through the angle $ \theta=\pi/n $ and then evenly reflect about the line $ \theta=\pi/n $. That is, in polar coordinates we set
\begin{equation*}
\tilde{\eta}_{tp}(r,\theta)=\tilde{\eta}_{bt}\left(r,\pi/n-\theta\right) \text{ and  } \tilde{\Phi}_{tp}(r,\theta)=\tilde{\Phi}_{bt}\left(r,\pi/n-\theta\right).
\end{equation*}

%
%
\section{Discussion}
In this paper we studied the convergence and scaling with $\tau$ for minimizers of the F\"{o}ppl-von K\'{a}rm\'{a}n energy for non-Euclidean plates with constant negative Gaussian curvature $ K_0 $. Specifically, to obtain a better understanding of the experimental results in  \cite{Shankar2011Gels}, we focused on annular domains and deformations with a periodic profile of $ n $ waves.

The first main result of this work is theorem \ref{perconfig:EnergyScalingTwo}, which gives rigorous upper and lower bounds for the elastic energy of minimizing deformations. This theorem rigorously proves that for all thickness values the $2$-wave saddle shape is energetically preferred over deformations with a higher number of waves. Specifically, this theorem proves that in the FvK approximation there can be no refinement of the number of waves with decreasing thickness in contrast to what is observed experimentally. 

The second main result is the scaling with $ \tau $ of the width of boundary layers in which stretching energy is concentrated. In these localized regions of stretching, the total elastic energy of an isometric immersion is lowered by allowing some stretching to reduce the bending energy. Near the edge of the annulus the contributions to the bending energy coming from the Gaussian and mean curvatures is reduced in overlapping boundary layers that scale with the thickness like $ t^{1/2}|K_0|^{-1/4} $. Furthermore, for $ n\geq 3$ boundary layers form along the lines of inflection in which the mean curvature of an isometric immersion is reduced by removing a jump discontinuity in the curvature along the azimuthal direction. The width of these boundary layers scales with the radius and thickness like $\rho^{1/3}t^{1/3}|K_0|^{-1/6} $. 

It is important to note that the FvK elastic energy is valid in an asymptotic limit in which the metric becomes increasingly flat with decreasing thickness, that is $ \epsilon\sim t/R $. In this paper we took $\epsilon$ fixed and $ t $ decreasing and in particular we showed that the energy of minimizers scales like $\tau^2$. Consequently, when $ \tau\ll 1 $ it is perhaps more appropriate to consider the Kirchhoff model. But, the Kirchhoff model is rigid in the sense that there is no competition between stretching and bending energies and thus there is no possibility of the Kirchhoff model alone explaining the refinement of the number of waves with decreasing thickness. 

It should also be pointed out that by the results of Hilbert \cite{Hilbert}, Holmgren \cite{Holmgren} and Amsler \cite{Amsler} concerning that non-existence of analytic isometric immersions of the hyperbolic plane into $ \mathbb{R}^3$, that $ E_{\rm{Ki}} $ diverges in $ R$ for fixed local isometric immersions. This is a geometric feature of non-Euclidean plates that is not captured in the linearized geometry of the FvK model. Indeed it was conjectured in \cite{Gemmer2011} that
\begin{equation*}
\max_{\mathbf{x}\in \mathcal{A}_{\rm{Ki}}\cap \mathcal{C}^{\infty}(\mathcal{D},\mathbb{R}^3)}\{|k_1|,|k_2|\}\geq \exp\left(\frac{\epsilon}{64}\right),
\end{equation*}
where $ k_1,k_2 $ are the principal curvatures of the surface $ \mathbf{x}(\mathcal{D}) $. This is in contrast with the FvK model in which $|D^2\eta|$ does not grow with the size of the domain. 

Furthermore, it was shown in \cite{Gemmer2011} that the $ n-$periodic isometric immersions in the FvK ansatz approximate exact isometric immersions with a periodic profile. The principal curvatures -- and thus the bending energy $ E_{\rm{Ki}} $ -- of these exact isometric immersions diverges at a finite radius $R_n$ that scales with $ n $ like $  R_n\sim \log(n)$ \cite{Gemmer2011}. This gives a geometric mechanism for the refinement of the number of waves with increasing radius of the disk but it does not explain the experimentally observed refinement with decreasing thickness.

From this work and \cite{Gemmer2011} it is clear that the FvK and Kirchhoff models of non-Euclidean elasticity cannot completely explain the periodic shapes in \cite{Shankar2011Gels}. The periodic shapes we have constructed in this paper are qualitatively similar to the experimental shapes but are only local minimizers of the FvK energy. A major goal of future works is to connect the existence of these local minimizers to the observed patterns in experiments. Below we outline two avenues of future research that could shed light on these issues.

First, the exact isometric immersions with periodic profile constructed in \cite{Gemmer2011} are similar to the $n$-periodic isometric immersions in the FvK but have highly localized regions of bending energy near the edge of the disk and along the lines of inflection. Therefore, in these regions, as in the boundary layers in the FvK ansatz, it is conceivable that it would be energetically favourable to allow some stretching to reduce this localized bending energy. These observations illustrate the multiple scale behaviour of this problem -- namely the scales $t/R$ and $\sqrt{|K_0|R}$ -- and it may be more appropriate to consider a combination of different reduced theories in various regions of the domain. A hierarchy of such reduced theories has recently been conjectured by Lewicka, Pakzad, and Mahadevan \cite{linearizedGeometry} and further research in this direction may explain the complex morphologies of non-Euclidean plates.

Second, the periodic shapes in swelling hydrogels are the result of dynamical processes. It may be more appropriate to model this type
of differential growth dynamically, perhaps as a gradient flow of the elastic energy. The pattern could then be selected for dynamical reasons and not by global minimization of an energy functional. This might explain why local but not global extrema for the energy functional
seem to describe the observed patterns.

\ack The authors wish to thank Efi Efrati for useful discussions regarding the scaling of boundary layers near the edges of the annuli. We would also like to thank L. Mahadevan for pointing out reference \cite{Lamb}. This work was supported by the US-Israel BSF grant 2008432, NSF grant DMS-0807501, and a VIGRE fellowship.
%
%
 
 %
 %

\section*{References}
\def\cprime{$'$}

\end{document}